\documentclass[reqno,10pt]{amsart}
\usepackage{graphicx,amsfonts,amssymb,amsmath,amsthm,url,amscd,comment}
\usepackage{color}
\usepackage{enumerate}

\usepackage{chngcntr}
\usepackage[usenames,dvipsnames]{xcolor}
\usepackage[normalem]{ulem}
\usepackage{pdfsync}
\usepackage{graphicx, amsmath, amssymb, amsfonts, amsthm, stmaryrd, tikz, amscd}
\usepackage[all]{xy}
\usetikzlibrary{matrix,arrows,decorations.pathmorphing}
\usepackage{graphicx, amsmath, amssymb, amsfonts, amsthm, stmaryrd, tikz, amscd}
\usepackage[all]{xy}
\usetikzlibrary{matrix,arrows,decorations.pathmorphing}

 \usepackage[hyperfootnotes=false, colorlinks, citecolor=   RoyalBlue, urlcolor=blue, linkcolor=blue         ]{hyperref}

\theoremstyle{plain} %--default
\newtheorem{theorem}             {Theorem} 

\newtheorem*{theoremA}             {Theorem A}
\newtheorem*{theoremB}             {Theorem B} 

\newtheorem{corollary}[theorem] {Corollary}

\newtheorem{conjecture}[theorem] {Conjecture}

\theoremstyle{definition}

\theoremstyle{plain} %--default
\newtheorem{proposition} [theorem] {Proposition}

\theoremstyle{remark}

\newtheorem{definition} [theorem]  {Definition}
\newtheorem*{definition*}  {Definition}

\newtheorem*{example*}    {Example}

\newtheorem{remark}  [theorem]           {Remark}
\newtheorem*{remark*}            {Remark}

\newtheoremstyle{itplain} % name
    {6pt}                    % Space above
    {5pt\topsep}                    % Space below
    {\itshape}                   % Body font
    {}                           % Indent amount
    {\itshape}                   % Theorem head font
    {.}                          % Punctuation after theorem head
    {5pt plus 1pt minus 1pt}                       % Space after theorem head
    {}  % Theorem head spec (can be left empty, meaning ‘normal’)

\theoremstyle{itplain} %--default
\newtheorem{lemma}[theorem]{Lemma}

\newtheorem*{lemma*}{Lemma}

\newtheorem*{corollary*} {Corollary} 

\theoremstyle{remark} %--default

\newtheorem*{lemmatest*}{Lemma}

\usepackage{etoolbox}
\patchcmd{\section}{\scshape}{\bfseries}{}{}
\makeatletter
\renewcommand{\@secnumfont}{\bfseries}
\makeatother

\renewcommand{\Re}{\mathrm{Re}}

\renewcommand{\geq}{\geqslant}
\renewcommand{\leq}{\leqslant}
\renewcommand{\mod}{\mathrm{mod}\,}

\numberwithin{equation}{section}
\numberwithin{theorem}{section}

\DeclareMathOperator{\SL}{SL}

\DeclareMathOperator{\GL}{GL}

\DeclareMathOperator{\SO}{SO}

\DeclareMathOperator{\ad}{ad}

\def\eps{\varepsilon}

\def\PB{\operatorname{PB}}
\def\PGL{\operatorname{PGL}}

\DeclareMathOperator{\trace}{trace}

\DeclareMathOperator{\asai}{asai}

\DeclareMathOperator{\Gal}{Gal}

\def\O{\operatorname{O}}

\DeclareMathOperator{\ram}{ram}

\DeclareMathOperator{\res}{res}

\DeclareMathOperator{\tr}{tr}

\DeclareMathOperator{\supp}{supp}

\author{Paul D. Nelson}
\address{ETH Z{\"u}rich, Department of Mathematics, R{\"a}mistrasse 101, CH-8092, Z{\"u}rich, Switzerland}
\email{paul.nelson@math.ethz.ch}

\date{\today}
\title{Quadratic Hecke sums and mass equidistribution}
\hypersetup{
  pdfkeywords={},
  pdfsubject={},
  pdfcreator={Emacs 24.5.1 (Org mode 8.2.10)}}
\begin{document}

\begin{abstract}
  We consider the analogue of the quantum unique ergodicity conjecture for
  holomorphic Hecke eigenforms on compact arithmetic hyperbolic
  surfaces.  We show that this conjecture follows from
  nontrivial bounds for Hecke eigenvalues summed over quadratic
  progressions.  Our reduction provides an analogue for the
  compact case of a criterion established by Luo--Sarnak for the
  case of the non-compact modular surface.  The novelty is that
  known proofs of such criteria have depended crucially upon
  Fourier expansions, which are not available in the compact
  case.  Unconditionally, we establish a twisted variant of the
  Holowinsky--Soundararajan theorem involving restrictions of
  normalized Hilbert modular forms arising via base change.
\end{abstract}

\maketitle

% Useful references from 2016-02-17 and 2015-11-17.

\setcounter{tocdepth}{1}
\tableofcontents

\section{Introduction}\label{sec:introduction}
\subsection{Context}
Let $B$ be a quaternion algebra over $\mathbb{Q}$.
We assume that $B$ splits
over $\mathbb{R}$, and fix an identification of the real
completion $B_\infty := B \otimes_{\mathbb{Q}} \mathbb{R}$ with
the matrix algebra $M_2(\mathbb{R})$.  Let $R$ be a maximal order in $B$.
Let $\mathbb{H}$ denote the upper half-plane.
Let $R^{(1)}$ denote the group of norm one units in $R$,
regarded as a subgroup of $\SL_2(\mathbb{R})$, and write
$\mathbf{Y} := R^{(1)} \backslash \mathbb{H}$ for the
corresponding finite volume arithmetic hyperbolic surface.

The space $\mathbf{Y}$ exhibits an important dichotomy according
to whether $B$ is \emph{split} (over $\mathbb{Q}$).  In the split
case,
we may identify $B$ with the $2 \times 2$ matrix algebra
$M_2(\mathbb{Q})$
and choose $R = M_2(\mathbb{Z})$, so that $R^{(1)} = \SL_2(\mathbb{Z})$.
The quotient
$\mathbf{Y} = \SL_2(\mathbb{Z}) \backslash \mathbb{H}$ is then
non-compact, and modular forms on $\mathbf{Y}$ enjoy
\emph{Fourier expansions} $\sum a_n e(n z)$
($e(z) := e^{2 \pi i z}$) corresponding to their invariance
under the substitution $z \mapsto z+1$ generating the stabilizer
of the cusp at $\infty$.  In the non-split case, the quotient
$\mathbf{Y}$ is compact, and such expansions are not available.
For analytic problems involving such quotients $\mathbf{Y}$, the split case is often more technically complicated due to the
non-compactness of $\mathbf{Y}$ and the continuous part of the
spectral decomposition of $L^2(\mathbf{Y})$, but this technical
complication is compensated for by the existence of Fourier
expansions, which have proven to be a useful analytic tool.
This work continues a series of works \cite{MR3356036, nelson-variance-73-2, nelson-variance-II, nelson-variance-3}
studying the non-split case of problems that had previously been
understood only in the split case by means of Fourier
expansions.

We turn to the main subject matter of this paper.  Lindenstrauss
\cite{MR2195133} and Soundararajan \cite{2009arXiv0901.4060S},
addressing a case of the quantum unique ergodicity
conjecture of Rudnick--Sarnak \cite{MR1266075},
showed that cuspidal Hecke--Laplace eigenfunctions on
$\mathbf{Y}$ have equidistributed $L^2$-mass in the large
eigenvalue limit.
We consider here the analogous
problem for holomorphic forms.  Let $(\varphi_k)_k$ be a
sequence, indexed by a sequence of large enough even integers
$k$, consisting of (nonzero) cuspidal holomorphic Hecke eigenforms
$\varphi_k$ on $\mathbf{Y}$ of weight $k$.
\begin{conjecture}\label{conj:hque-compact}
  For each bounded continuous $\Psi : \mathbf{Y} \rightarrow
  \mathbb{C}$,
  we have
  \begin{equation}\label{eqn:hque-conj}
    \frac{
      \int_{\mathbf{Y}}
      y^k
      |\varphi_k(z)|^2
      \Psi(z)
      \, \frac{d x \, d y}{y^2}
    }
    {
      \int_{\mathbf{Y}}
      y^k
      |\varphi_k(z)|^2
      \, \frac{d x \, d y}{y^2}
    }
    \rightarrow 
    \frac{
      \int_{\mathbf{Y}}
      \Psi(z)
      \, \frac{d x \, d y}{y^2}
    }
    {
      \int_{\mathbf{Y}}
      \, \frac{d x \, d y}{y^2}
    }
  \end{equation}
  as $k \rightarrow \infty$.
\end{conjecture}

By Watson's formula \cite{watson-2008}, this conjecture follows
from the generalized Lindel{\"o}f hypothesis (or indeed, from
any subconvex bound for the associated family of triple product
$L$-functions \eqref{eq:triple-product-finite-part-intro}).
Sarnak \cite{Sar01} established the special case of this
conjecture in which the $\varphi_k$ are dihedral.  The general
split case (formulated by Luo--Sarnak \cite{luo-sarnak-mass}) is
a celebrated result of Holowinsky--Soundararajan
\cite{MR2680499}.  The general non-split case remains open.

The work of
Holowinsky--Soundararajan synthesizes two complementary methods
developed independently by Soundararajan \cite{soundararajan-2008} and
Holowinsky \cite{MR2680499}.  The method of Soundararajan
applies just as well to the non-split case, while the
method of Holowinsky does not.  The latter method departs
by reducing the problem to suitable
estimates for shifted convolution sums
\begin{equation}\label{eqn:shifted-sums}
  \sum_n
  f(n/k)
  \lambda_{\varphi_k}(n)
  \lambda_{\varphi_k}(n + \ell).
\end{equation}
Here $f \in C_c^\infty(\mathbb{R}^\times_+)$ is a fixed test
function, $\ell$ is a fixed integer, and
$\lambda_{\varphi_k} : \mathbb{N} \rightarrow \mathbb{C}$
describes the Hecke eigenvalues of $\varphi_k$, normalized so that
the Deligne bound reads $|\lambda_{\varphi_k}(p)| \leq 2$ for primes $p$.
The following criterion, established by Luo--Sarnak \cite[Cor
2.2]{luo-sarnak-mass}, clarifies the relationship
between Conjecture
\ref{conj:hque-compact} and bounds for such sums.
\begin{theorem}[Luo--Sarnak]\label{thm:luo-sarnak-criterion}
  Assume that $B$ is split.
  Then
  Conjecture
  \ref{conj:hque-compact}
  holds
  if and only if
  for each $f \in C_c^\infty(\mathbb{R}_+^\times)$
  and $\ell \in \mathbb{Z}$,
  \begin{equation}\label{eq:luo-sarnak-split-sums}
    \lim_{k \rightarrow \infty}
    \frac{\zeta(2)}{k L(\ad \varphi_k, 1)}
    \sum_{n \in \mathbb{N}}
    f(n/k)
    \lambda_{\varphi_k}(n)
    \lambda_{\varphi_k}(n + \ell)
    =
    1_{\ell=0}
    \int_{0}^{\infty} f(y) \, d y.
  \end{equation}
\end{theorem}
Here and henceforth $L(\dotsb,s)$ denotes the finite part of an
$L$-function, excluding archimedean factors.

The proofs of Theorem \ref{thm:luo-sarnak-criterion} and its
variant due to Holowinsky \cite[Thm 1]{MR2680499} are based upon
an analysis of Fourier expansions and the associated Poincar{\'e} series.
Indeed,
for a certain Poincar{\'e} series $\Psi$ attached to $\ell$ and $f$,
the left hand sides of \eqref{eq:luo-sarnak-split-sums}
and \eqref{eqn:hque-conj} are asymptotic,
while the right hand sides are equal.
In particular, those proofs
are fundamentally limited to the split case.
It is natural to ask
whether some criterion analogous to Theorem
\ref{thm:luo-sarnak-criterion} might exist in the non-split case.
As a first hint, we note that by
Hecke multiplicativity and M{\"o}bius inversion, estimates for
the shifted sums \eqref{eqn:shifted-sums} are substantially
equivalent to those for the expressions
\begin{equation}\label{eqn:quadratic-sums-hecke-eigenvalues}
  \sum_n
  f(n/k)
  \lambda_{\varphi_k}(Q(n))
\end{equation}
when $Q$ is a \emph{reducible} quadratic polynomial of the form
$Q(n) = n(n + \ell)$.
By analogy, one might speculate that the non-split case
of
Conjecture \ref{conj:hque-compact}
should be
related, somehow, to the sums
\eqref{eqn:quadratic-sums-hecke-eigenvalues} for
\emph{irreducible} quadratic polynomials $Q$.  It is perhaps
less clear how one might prove such a relationship.

\subsection{Results}
The main purpose of this article is to confirm the speculation
indicated above.
Informally, our first main result reduces the non-split case of
Conjecture \ref{conj:hque-compact} to upper bounds for
\eqref{eqn:quadratic-sums-hecke-eigenvalues} of the shape
\begin{equation}\label{eqn:informal-estimate-description}
  o_{k \rightarrow \infty}(k L(\ad \varphi_k,1)),
\end{equation}
for irreducible $Q$,
with ``polynomial dependence'' upon $Q$ and $f$.

\begin{definition}
  By an \emph{integer-valued quadratic polynomial} we mean a
  polynomial $Q \in \mathbb{Q}[x]$ of the form
  $Q(x) = a x^2 + b x + c$, with $a \neq 0$, such that
  $Q(n) \in \mathbb{Z}$ for all $n \in \mathbb{Z}$.
  Such a polynomial is irreducible precisely
  when it has no rational roots.
  We set
  $\|Q\| := \max(|a|,|b|,|c|)$.
\end{definition}

\begin{definition}\label{defn:sobolev-norms}
  For $f \in C_c^\infty(\mathbb{R}^\times_+)$, we write $f'$
  for the derivative, $D f(y) := y f'(y)$ for the invariant
  derivative, and $\mathcal{S}_N(f)$ for
  the Sobolev norms
  defined for
  $N \in \mathbb{Z}_{\geq 0}$ by the formula
  \begin{equation}\label{eqn:}
    \mathcal{S}_N(f)
    :=
    \max _{j \leq N} \sup_{y \in \mathbb{R}^\times_+}
    (y + 1/y)^N |D^j f (y)|.
  \end{equation}
  \label{defn:}
\end{definition}

We recall a special case of a result of
Templier--Tsimerman \cite[Thm 1]{MR3096570},
generalizing and extending earlier work of Blomer
\cite{MR2435749}
and Templier \cite{MR2473297}.
\begin{theorem}[Templier--Tsimerman]\label{thm:templier-tsimerman}
  Let $\pi$ be a non-dihedral
    cuspidal automorphic
  representation of $\GL_2$.  Write
  $L(\pi,s) = \sum_{n \geq 1} \lambda_{\varphi}(n)/n^s$.  Let $Q$ be a
  non-split polynomial of the form $Q(n) = n^2 - d$ for some
  (non-zero, non-square) integer $d$.
  Assume that $f$ is supported in the interval $(1,2)$
  and that $\mathcal{S}_N(f) \ll_N 1$ for each fixed $N$.
  Then for $\eps > 0$ and $X \gg |d|^{1/2}$,
  \begin{equation}\label{eqn:templier-tsimerman}
    \sum_{n}
    \lambda_{\varphi}(Q(n)) f(n/X)
    \ll_{\varphi, \eps}
    X^{1/2+\eps} |d|^{\delta}
    \left(\frac{X}{|d|^{1/2}} \right)^{\vartheta}.
  \end{equation}
\end{theorem}
The exponents $\vartheta, \delta$ are described in \cite[Thm
1]{MR3096570}; they quantify bounds towards the Ramanujan
conjecture in both integral and half-integral weight, for which
the latest records are $\vartheta \leq 7/64$ \cite{MR1937203,
  MR2811610} and $\delta \leq 1/6$
\cite{MR1779567, MR3394377, Nelson-EisCubic}.

Because the estimate \eqref{eqn:templier-tsimerman} depends in
an unspecified manner upon the eigenform $\varphi$, it does not
apply to sums like
\eqref{eqn:quadratic-sums-hecke-eigenvalues} in which the
eigenform and the length of summation vary simultaneously.
We
might nevertheless expect, in the absence of obvious
evidence to the contrary, that the sums
\eqref{eqn:quadratic-sums-hecke-eigenvalues} enjoy square-root
cancellation.
In particular, we might extrapolate Theorem
\ref{thm:templier-tsimerman} to the following
conjecture.
\begin{conjecture}\label{conj:non-split-sums}
  Assume that the $\varphi_k$ are non-dihedral.\footnote{ In
    fact, since we have restricted to $\varphi_k$ of full level,
    the non-dihedrality assumption is automatic.  The conjecture
    generalizes to the case of higher levels, where such an
    assumption is relevant.  } There exists $A \geq 0$ so that
  for all sequences of irreducible integer-valued quadratic
  polynomials $Q_k$ and test functions
  $f_k \in C^\infty(\mathbb{R}^\times_+)$ satisfying
  $\mathcal{S}_N(f_k) \ll_N 1$ for each fixed $N$,
  \begin{equation}\label{eqn:non-split-sums-conj}
    \lim_{k \rightarrow \infty}
    \frac{
      \sum _{n}
      \lambda_{\varphi_k}(|Q_k(n)|) f_k(n/k)
    }{
      k L(\ad \varphi_k, 1)
      \|Q_k\|^A
    }
    = 0.
  \end{equation}
\end{conjecture}
The quantities $Q_k$ and $f_k$ in
Conjecture \ref{conj:non-split-sums}
are allowed to vary
with $k$, but standard estimates for $|\lambda_{\varphi}|$ and
$L(\ad \varphi_k, 1)$ show that
the desired estimate \eqref{eqn:non-split-sums-conj}
is nontrivial only if they vary at most mildly.
Conjecture \ref{conj:non-split-sums} should thus be regarded as
a mild strengthening of the true analogue of
\eqref{eq:luo-sarnak-split-sums} obtained by taking $Q_k$ and
$f_k$ independent of $k$.
For
instance, the Deligne bound
$|\lambda_{\varphi}(n)| \leq \tau(n)$ (here $\tau$ denotes the
divisor function), the divisor bound
$\tau(n) \ll_{\eps} n^{\eps}$ and the Hoffstein--Lockhart bound
$L(\ad \varphi_k, 1) \gg_{\eps} k^{-\eps}$
give what one might call the trivial bound
\begin{equation}\label{eqn:trivial-bound}
  \frac{
    \sum _{n}
    \left\lvert
      \lambda_{\varphi_k}(|Q_k(n)|) f_k(n/k)
    \right\rvert
  }{
    k L(\ad \varphi_k, 1)
  } \ll_{\eps} k^\eps.
\end{equation}
Using \eqref{eqn:trivial-bound}, or even the much weaker
estimate obtained by replacing $k^\eps$ with $k^{O(1)}$, we may
reduce the verification of \eqref{eqn:non-split-sums-conj} to
the case that for all $\eps > 0$,
\[
  \|Q_k\| \ll_{\eps} k^{\eps},
  \quad 
  \supp(f_k) \subseteq [k^{-\eps}, k^{\eps}] \text{ for } k \geq k_0(\eps).
\]
We may reduce further using the following
more precise estimate:

\begin{theorem}[Sieve bound]\label{thm:sieve-bound}
  Under the hypotheses of
  Conjecture \ref{conj:non-split-sums},
  we have
  \begin{equation}\label{eqn:sieve-bound-second-estimate}
  \frac{
    \sum _{n}
    \left\lvert
      \lambda_{\varphi_k}(|Q_k(n)|) f_k(n/k)
    \right\rvert
  }{
    k L(\ad \varphi_k, 1)
  } \ll_{\eps}
  \|Q_k\|^{\O(1)} (\log k)^{1/2+\eps},
\end{equation}
where all implied constants
are absolute.
\end{theorem}

The proof is recorded in \S\ref{sec:decay-norm-base}, following
that of Proposition \ref{prop:twisted-holowinsky}.
This bound
allows us to reduce
the verification of \eqref{eqn:non-split-sums-conj} to
the case that
\begin{equation}\label{eqn:}
  \|Q_k\| \ll_{\eps} (\log k)^{\eps},
  \quad 
  \supp(f_k) \subseteq [(\log k)^{-\eps}, (\log
  k)^{\eps}] \text{ for } k \geq k_0(\eps).
\end{equation}

We do not know how to improve upon
\eqref{eqn:sieve-bound-second-estimate}
unconditionally.
We expect the following  stronger estimate.
\begin{conjecture}\label{conj:strong-abs-value-bound}
  There exists $\delta > 0$
  so that
  under the hypotheses of
  Conjecture \ref{conj:non-split-sums},
  \begin{equation}\label{eqn:sieve-bound-second-estimate-2}
    \frac{
      \sum _{n}
      \left\lvert
        \lambda_{\varphi_k}(|Q_k(n)|) f_k(n/k)
      \right\rvert
    }{
      k L(\ad \varphi_k, 1)
    } \ll_{\eps}
    \|Q_k\|^{\O(1)} (\log k)^{-\delta}.
  \end{equation}
\end{conjecture}
Conjecture \ref{conj:strong-abs-value-bound}
is motivated
by hypothetical uniform forms of the Sato--Tate
conjecture (see Remark
\ref{rmk:sufficiently-uniform-sato-tate} and \cite{MR2484279}).  Clearly Conjecture
\ref{conj:strong-abs-value-bound} implies Conjecture
\ref{conj:non-split-sums}.

We may now state our first main result:
\begin{theoremA}
  \hypertarget{thm:impl}{}
  Conjecture \ref{conj:non-split-sums}  implies  Conjecture \ref{conj:hque-compact}.
\end{theoremA}

\begin{remark}
  It should be possible to formulate an analogue of Conjecture
  \ref{conj:non-split-sums} in dihedral cases by incorporating a
  suitable main term (as in \cite{MR3096570}) and in reducible
  cases by adapting \eqref{eq:luo-sarnak-split-sums}.
\end{remark}

\begin{remark}
  The proof of Theorem
  \hyperlink{thm:impl}{A}
  shows that to deduce
  Conjecture \ref{conj:hque-compact}, it is not necessary to
  know that \eqref{eqn:non-split-sums-conj} holds for
  \emph{every} $Q := Q_k$, but rather for ``sufficiently many'' $Q$
  satisfying the condition
  \begin{equation}\label{eqn:local-splitting-condition-intro}
    \text{$Q$ is irreducible
      at every place at which $B$
      does not split},
  \end{equation}
  and no further splitting conditions.  For instance, when $B$
  is split (over $\mathbb{Q}$), the condition
  \eqref{eqn:local-splitting-condition-intro} is empty and it
  suffices to consider reducible $Q$, as follows from the
  Luo--Sarnak criterion.  When $B$ is non-split, the condition
  \eqref{eqn:local-splitting-condition-intro} forces $Q$ to be
  irreducible.
\end{remark}

\begin{remark}\label{rmk:expected-generalizations}
  We expect that Conjecture \ref{conj:non-split-sums} also
  implies the generalization of Conjecture
  \ref{conj:hque-compact} to higher fixed level (e.g., taking
  for $R$ an Eichler order) and to definite quaternion algebras
  (as in the ``QUE on the sphere'' problem considered in
  \cite{MR2018269}).  Conversely, we expect that mildly
  generalized and strengthened forms of Conjecture
  \ref{conj:hque-compact} and Conjecture
  \ref{conj:non-split-sums} are equivalent, but we do not
  attempt to formulate such an equivalence here.  The analogue
  of Conjecture \ref{conj:non-split-sums} for the ``level
  $q \rightarrow \infty$'' aspect as in
  \cite{PDN-HQUE-LEVEL,MR3110797,MR3801500} should involve
  quadratic polynomials that vary considerably with $q$, e.g.,
  $n \mapsto n^2 - q^2 \ell$.
\end{remark}

\subsection{Applicability of the Holowinsky--Soundararajan
  method}
We address here the tantalizing question of whether Theorem \hyperlink{thm:impl}{A}
and the Holowinsky--Soundararajan method suffice to resolve
Conjecture \ref{conj:hque-compact}.
We will observe a significant discrepancy between the split and non-split cases,
arising
ultimately
from dichotomies
of
the following sort:
\begin{itemize}
\item $n(n+1)$ almost always has at least two prime divisors, but
\item $n^2 + 1$ is expected to be prime infinitely often. 
\end{itemize}

Soundararajan's results
(see \cite{soundararajan-2008},
\cite[Lem 2]{MR2680499}
and \cite[\S1.3]{PDN-HMQUE})
show that
for $\Psi$ a cuspidal Hecke--Maass form,
the conclusion of
Conjecture \ref{conj:hque-compact} holds
provided that the Hecke eigenvalues
$\lambda := \lambda_{\varphi_k}$ satisfy the estimate
\begin{equation}\label{eqn:sound-sufficient-condition}
  \frac{
    \sum _{p \leq k} |\lambda(p)|^2/p
  }{
    \sum _{p \leq k} 1/p
  }
  \geq 1/2 + \delta
  \quad
  \text{for some
    fixed $\delta > 0$}
\end{equation}
for large enough $k$.
In seeking to prove Conjecture \ref{conj:hque-compact} for such $\Psi$,
we may thus assume without loss of generality
(after passing to a subsequence if necessary) that
the condition \eqref{eqn:sound-sufficient-condition}
fails.  By the triangle inequality for the $\ell^2$-norm, we
then have
\begin{equation}\label{eqn:sound-sufficient-condition-2}
  \frac{
    \sum _{p \leq k} (1-|\lambda(p)|)^2/p
  }{
    \sum _{p \leq k} 1/p
  }
  \geq \delta
  \quad 
  \text{ for some fixed $\delta > 0$,}
\end{equation}
or indeed, for any fixed $\delta < (1 - 1/\sqrt{2})^2$.

\begin{remark}
  The condition \eqref{eqn:sound-sufficient-condition} is
  expected to hold, since a sufficiently uniform form of the
  Sato--Tate conjecture would imply that
  $\sum_{p \leq k} |\lambda(p)|^2 / p \sim \sum _{p \leq k } 1
  /p$, but this expectation seems difficult to establish
  unconditionally.
\end{remark}

We now recall how Holowinsky's approach \cite{MR2680498}
establishes the Luo--Sarnak criterion
\eqref{eq:luo-sarnak-split-sums} for $\ell \neq 0$
under the assumption \eqref{eqn:sound-sufficient-condition-2}.  (The case
$\ell = 0$ requires an additional ``$Y$-thickening'' technique
(see \cite[\S3.1]{MR2680498}
or \cite[Lem 5.4]{MR3356036})
which
we do not discuss here.)  Holowinsky bounds the Hecke
eigenvalues in magnitude, forfeiting any potential cancellation
in the sums \eqref{eqn:shifted-sums}, and appeals to
sieve-theoretic bounds.
For simplicity, take $\ell = 1$.
We must verify that
\begin{equation}\label{eqn:holowinsky-target}
  \frac{\sum _{n \leq k} |\lambda(n) \lambda(n+1)|}{k L(\ad \varphi_k,1)}
\end{equation}
tends to zero
as $k \rightarrow \infty$.
On the one hand, it follows from
\cite[Lem 2]{MR2680499} that
$L(\ad \varphi_k,1)$ is bounded
from below (possibly up to a $(\log \log k)^{\O(1)}$
factor, negligible for the present aims) by
$\exp \sum_{p \leq k} (|\lambda(p)|^2 - 1)/p$.
On the other
hand, a sieve bound
due to Nair \cite{MR1197420}
gives the estimate
\begin{equation}\label{eqn:shifted-sums-with-abs}
  \frac{1}{k}
  \sum _{n \leq k}
  |\lambda(n) \lambda(n+1)|
  \ll
  \exp \sum _{p \leq k}
  \frac{2 |\lambda(p)| - 2}{p}.
\end{equation}
Thus \eqref{eqn:holowinsky-target}
is majorized by
\begin{equation}\label{eqn:holowinsky-majorize-negative-stuff}
  \frac{
    \exp \sum _{p \leq k}
    (2 |\lambda(p)| - 2)/p
  }{
  \exp \sum_{p \leq k} (|\lambda(p)|^2 - 1)/p
  }
  =
  \exp
  \left(
    - \sum _{p \leq k}
    \frac{(1 - |\lambda(p)|)^2}{p}
 \right).
\end{equation}
If \eqref{eqn:sound-sufficient-condition-2} holds,
then \eqref{eqn:holowinsky-majorize-negative-stuff}
decays.
Holowinsky--Soundararajan \cite{MR2680499}
established the split case
of Conjecture
\ref{conj:hque-compact} via similar arguments.

To establish the non-split case of Conjecture
\ref{conj:hque-compact} via Theorem \hyperlink{thm:impl}{A} and the
Holowinsky--Soundararajan method would seem to require verifying
that
if
\eqref{eqn:sound-sufficient-condition} fails,
then the non-split shifted convolution sums
such as
\begin{equation}\label{eqn:holowinsky-target-2}
  \frac{\sum _{n \leq k} |\lambda(n^2 + 1)|}{k L(\ad \varphi_k,1)}
\end{equation}
tend to zero as $k \rightarrow \infty$.
The sieve bound \cite{MR1197420}
analogous to \eqref{eqn:shifted-sums-with-abs}
reads
\begin{equation}\label{eqn:shifted-sums-with-abs-2}
  \frac{1}{k}
  \sum _{n \leq k}
  |\lambda(n^2+1)|
  \ll
  \exp \sum _{p \leq k : p \equiv 1(4)}
  \frac{2 |\lambda(p)| - 2}{p},
\end{equation}
so \eqref{eqn:holowinsky-target-2} is majorized by
the ratio
\begin{equation}\label{eqn:}
  \frac{
    \exp \sum _{p \leq k : p \equiv 1(4)}
    (2 |\lambda(p)| - 2)/p
  }{
  \exp \sum_{p \leq k} (|\lambda(p)|^2 - 1)/p
},
\end{equation}
which we may rewrite up to bounded multiplicative error as
\begin{equation}\label{eq:nonsplit-holowinsky-target}
  \exp
  \left(
    - \sum _{p \leq k : p  \equiv 1(4)}
    \frac{(1 - |\lambda(p)|)^2}{p}
    +  \sum_{p \leq k : p \equiv 3(4)}
    \frac{1 - |\lambda(p)|^2}{p}
\right).
\end{equation}
Unfortunately, we see no way to deduce that such expressions
decay.
For instance,
we see no way to rule out unconditionally that
for $p \leq k$,
\begin{equation}\label{eqn:}
  |\lambda(p)| \approx  1
  \text{ for } p \equiv 1(4),
  \quad 
  |\lambda(p)| \approx  0
  \text{ for } p \equiv 3(4),
\end{equation}
in which case
\eqref{eqn:sound-sufficient-condition} fails
and
\eqref{eq:nonsplit-holowinsky-target} does not decay.
Even if the statistical
behavior of $\lambda(p)$ is sufficiently unbiased by the residue
class of $p$ modulo $4$
that
we may approximate
\eqref{eq:nonsplit-holowinsky-target}
by
\begin{equation}\label{eq:nonsplit-holowinsky-target-2}
  \exp
  \left(
    \sum _{p \leq k}
    \frac{- (1 - |\lambda(p)|)^2 + 1 - |\lambda(p)|^2}{ 2p}
  \right)
  =
  \exp
  \left(
    \sum _{p \leq k}
    \frac{ |\lambda(p)| - |\lambda(p)|^2}{ p}
  \right),
\end{equation}
then there remain hypothetical problem cases
such as when
\begin{equation}\label{eqn:}
  |\lambda(p)| \approx 1/\sqrt{2} \text{ for } p \leq k.
\end{equation}
Thus some new idea seems necessary to establish Conjecture
\ref{conj:hque-compact} via Theorem \hyperlink{thm:impl}{A}.

\begin{remark}
  Holowinsky's approach differs significantly from that of
  most literature on the shifted convolution problem
  (split or non-split), see for instance
  \cite[\S4.4]{MR2331346},
  Theorem \ref{thm:templier-tsimerman}
  and \cite{MR2437682, MR2435749, MR2473297, MR3096570}.  The
  cited works seek to achieve power savings estimates for sums
  like \eqref{eqn:shifted-sums} and
  \eqref{eqn:quadratic-sums-hecke-eigenvalues} by exploiting
  cancellation coming from the variation of the sign of the
  Hecke eigenvalues, but with the automorphic form $\varphi_k$
  held essentially fixed (i.e., independent of $k$) as the
  length of the sum increases.  The estimates obtained in this way are not
  sufficiently uniform with respect to $\varphi_k$ to broach the
  Luo--Sarnak criterion \eqref{eq:luo-sarnak-split-sums} or
  Conjecture \ref{conj:non-split-sums}.
\end{remark}

\begin{remark}\label{rmk:sufficiently-uniform-sato-tate}
  A sufficiently uniform
  form of the Sato--Tate conjecture
  would suggest that
  for non-dihedral $\varphi_k$,
  the collection of Hecke eigenvalues
  $\{\lambda(p) :  p \leq k \}$
  behaves like the random variable
  $2 \cos \theta$,
  where $\theta$ is sampled
  from $[0,\pi]$
  with respect to the probability measure
  $\frac{2}{ \pi } \sin^2 \theta \, d \theta$.
  In particular,
  we expect that
  \[
    \sum _{p \leq k}
    \frac{ |\lambda(p)| - |\lambda(p)|^2}{ p}
    \approx
    c
    \log \log k,
    \quad
    c :=
    \frac{2}{ \pi }
    \int _{0}^\pi
    \left(
      |2 \cos \theta|
      - 
      |2 \cos \theta|^2
    \right)
    \, \sin^2 \theta
    \, d \theta
    =
    \frac{8}{3 \pi } - 1
    \approx -0.151.
  \]
  Since $c < 0$, we do in fact expect the non-split sums
  \eqref{eqn:holowinsky-target-2} to decay (like
  $(\log k)^{c} \approx (\log k)^{-0.151}$), but unlike for the
  split sums \eqref{eqn:holowinsky-target},
  there is no apparent way to verify anything
  approaching this expectation.
  Compare with
  \cite{MR2484279}.
\end{remark}

\subsection{Main ideas of the proof}\label{sec:main-ideas-proof}
We now discuss the proof of Theorem \hyperlink{thm:impl}{A}.  We may
assume that $B$ is non-split and that $\Psi$ is a cuspidal
(i.e., non-constant)
Hecke--Maass eigenform.  The basic difficulty, relative to
existing methods, is that the automorphic forms appearing in the
integral on the LHS of \eqref{eqn:hque-conj} do not admit
Fourier expansions.  We aim to relate those integrals to other
integrals of automorphic forms that \emph{do} admit
Fourier
expansions.  This can be achieved using the
theta correspondence as in \cite{MR3356036, nelson-variance-73-2, nelson-variance-II, nelson-variance-3},
but since we are concerned here only with the
\emph{magnitude} of the integrals, it is more direct to work
with $L$-functions and period formulas.  Watson's triple product
formula relates the squared magnitude of the LHS of
\eqref{eqn:hque-conj} to the central triple product $L$-value
\begin{equation}\label{eq:triple-product-finite-part-intro}
  L(\varphi \times \varphi \times \Psi, 1/2),
\end{equation}
so our task is to estimate that $L$-value in terms of integrals
of automorphic forms that admit Fourier expansions and then to
relate such integrals to the Hecke eigenvalues of our original forms.

Naively, one might hope to achieve this aim by simply replacing
$\varphi$ and $\Psi$ by their Jacquet--Langlands lifts to the
$\PGL_2(\mathbb{Q})$, but then Prasad's uniqueness theorem
\cite[Thm 1.2]{MR1059954} implies that the corresponding triple
product integrals vanish identically for local reasons, hence
carry no information about the $L$-value
\eqref{eq:triple-product-finite-part-intro}.
(We discuss this point
at more length in Remark \ref{rmk:summary-thus-far}.)
We must thus look
outside the triple product setting.

We indicate two approaches to the problem.  The first approach
is not developed in detail in this the paper, but seems to us to
convey most efficiently that Conjectures \ref{conj:hque-compact}
and \ref{conj:non-split-sums} should be connected at all;
it is related to our existing work  \cite{MR3356036, nelson-variance-73-2,
  nelson-variance-II, nelson-variance-3}.
The
second approach is the basis of our proof of Theorem
\hyperlink{thm:impl}{A} (and also of Theorem
\hyperlink{thm:twisted-holowinsky-sound}{B}, stated below).
We expect that the two approaches
may be related
to one another 
via a seesaw identity as in \cite[Prop 5.2]{MR2198222} or
\cite[Proof of Thm 1]{MR3291638}.

\begin{enumerate}
\item
  By the factorization
  \begin{equation}\label{eq:intro-factorization-1}
    L(\varphi \times \varphi \times \Psi, 1/2)
    =
    L(\ad \varphi \times \Psi, 1/2) L(\Psi,1/2).
  \end{equation}
  it suffices to estimate the $L$-value
  $L(\ad \varphi \times \Psi, 1/2)$.  That $L$-value appears in
  Shimura-type integral representations on the metaplectic
  double cover of $\SL_2$ (see \cite[Thm 4.5]{MR3291638} and
  \cite{nelson-subconvex-reduction-eisenstein})
  roughly of the shape
  \begin{equation}\label{eq:metaplectic-triple-product}
    L(\ad \varphi \times \Psi, 1/2)
    \approx 
    |\int \overline{\varphi_k'} \theta \tilde{\Psi}|^2,
  \end{equation}
  where \begin{itemize}
  \item  $\varphi_k '(z) = \sum _{m \geq 1}
    \lambda_{\varphi_k}(m)
    m^{(k-1)/2}
    e(m z)$ denotes the Jacquet--Langlands lift
    of $\varphi_k$ to $\PGL_2$,
  \item $\theta$ is an elementary theta function, e.g.,
    $\theta(z) = \sum_{n \in \mathbb{Z}} e(n^2 z)$, and
  \item $\tilde{\Psi}$ is the Maass--Shintani--Waldspurger theta
    lift of $\Psi$.
  \end{itemize}
  Since $\tilde{\Psi}$ is fixed, it suffices to estimate the
  corresponding integrals obtained by replacing $\tilde{\Psi}$
  by a Poincar{\'e} series.  Those integrals unfold naturally in
  terms of the Fourier coefficients of the product
  \begin{equation}\label{eq:}
    \overline{\varphi_k'(z)} \theta(z)
    =
    \sum _{m \geq 1 }
    \lambda_{\varphi_k}(m)
    m^{(k-1)/2}
    e(-m \bar{z})
    \sum _{n \in \mathbb{Z} }
    e(n^2 z).
  \end{equation}
  The $\ell$th Fourier coefficient of that product has the shape
  \begin{equation}\label{eq:}
    \sum_{n} \lambda_{\varphi_k}(n^2-\ell) f(n).
  \end{equation}
  It in fact suffices to consider the restricted class of
  Poincar{\'e} series indexed by non-square $\ell$
  together with the elementary theta functions.
  We thereby encounter sums roughly as in  Conjecture
  \ref{conj:non-split-sums}.
  To implement this strategy
  would require an analysis of the test vector problem
  for the integrals \eqref{eq:metaplectic-triple-product},
  which we do not address here.

  \begin{remark}
    By applying the steps outlined here in reverse,
    it may be possible to establish a ``converse'' to Theorem
    \hyperlink{thm:impl}{A} of the sort suggested in Remark
    \ref{rmk:expected-generalizations}.
\end{remark}

\item Let $D$ be a non-square quadratic fundamental
  discriminant.  We assume, to eliminate some case analysis,
  that $D$ is positive.  Let $(\varphi_k)_D$ denote the weight
  $(k,-k)$ Hilbert modular form for
  $\PGL_2(\mathbb{Q} (\sqrt{D}))$ obtained from $\varphi_k$ by
  quadratic base change (see \S\ref{sec:base-change} for
  details).  By restriction, it defines a modular form
  $\res((\varphi_k)_D)$ for $\PGL_2(\mathbb{Q})$.  Let $\Psi '$
  denote the Jacquet--Langlands lift of $\Psi$ to a newform on
  $\PGL_2(\mathbb{Q})$.  Then the twisted triple product formula
  (\S\ref{sec:twist-triple-prod}) relates the squared
  restriction period
  \[|\int \res((\varphi_k)_D) \Psi '|^2
  \]
  to the twisted Asai $L$-value
  $L(\asai((\varphi_k)_D) \times \Psi,1/2)$, which factors as
  $L(\ad(\varphi_k) \times \Psi,1/2) L(\Psi \otimes
  \chi_D,1/2)$.  Crucially, we may choose $D$ so that the
  proportionality constant in this period formula is nonzero
  and so that $L(\Psi \otimes \chi_D,1/2)$ is nonzero
  (\S\ref{sec:choice-quadr-field}).
  We thereby reduce Conjecture
  \ref{conj:hque-compact} to suitable bounds for restriction
  periods.  By an ``approximate functional equation'' for
  such periods (\S\ref{sec:appr-funct-equat}), we may relate
  them to quadratic sums of Hecke
  eigenvalues (\S\ref{sec:asympt-arch-integr},
  \S\ref{sec:completion-proof}), hence to Conjecture \ref{conj:non-split-sums}.
\end{enumerate}

\subsection{A twisted variant of arithmetic quantum unique
  ergodicity}\label{sec:twist-vari-arithm}
We now describe our second main result, whose relation
to Theorem \hyperlink{thm:impl}{A} will be clarified below.

Fix a cuspidal Hecke--Maass eigenform $\Psi$ on
$\SL_2(\mathbb{Z}) \backslash \mathbb{H}$ and a real quadratic
field $\mathbb{Q}(\sqrt{D})$ of discriminant $D$.  For even
$k \geq 12$, let $\varphi_k$ be a (non-dihedral)
cuspidal holomorphic Hecke eigenform on
$\SL_2(\mathbb{Z})$ of weight $k$.  Let $(\varphi_k)_D$,
as in the proof sketch above,
denote
the quadratic base change lift of $\varphi_k$ to a weight
$(k,-k)$ cuspidal Hilbert modular newform on
$\PGL_2(\mathbb{Q}(\sqrt{D}))$,
(\S\ref{sec:base-change}).  We define the $L^2$-norm $\|(\varphi_k)_D\|$ by
integrating over
$\PGL_2(\mathbb{Q}(\sqrt{D})) \backslash
\PGL_2(\mathbb{A}_{\mathbb{Q}(\sqrt{D})})$ with respect to some
Haar measure.  (The measure normalization is not important for our purposes.)
By restriction, we obtain a function $\res((\varphi_k)_D)$ on
$\SL_2(\mathbb{Z}) \backslash \mathbb{H}$.  We define the
restriction period $\int \res((\varphi_k)_D) \Psi$ by
integrating over $\SL_2(\mathbb{Z}) \backslash \mathbb{H}$ with
respect to the standard measure $\frac{d x \, d y}{y^2}$.

\begin{theoremB}
  \hypertarget{thm:twisted-holowinsky-sound}{}
  For fixed $\eps> 0$, we have
  \begin{equation}\label{eqn:twisted-holowinsky-sound}
    \frac{\int  \res((\varphi_k)_D) \Psi}{\|(\varphi_k)_D\|}
    \ll_{\eps}
    (\log k)^{-1/8+\eps}.
  \end{equation}
\end{theoremB}
Note that if we apply the same construction but with $D=1$ and
$\mathbb{Q}(\sqrt{D})$ replaced by
$\mathbb{Q} \times \mathbb{Q}$, then
$\res((\varphi_k)_D)(z) = y^k |\varphi_k|^2(z)$ and the left
hand sides of \eqref{eqn:hque-conj} and
\eqref{eqn:twisted-holowinsky-sound} coincide for suitable
measure normalizations.  Thus the estimate
\eqref{eqn:twisted-holowinsky-sound} may be understood as a
twisted variant of the result \cite[Thm 1 (i)]{MR2680499} of
Holowinsky--Soundararajan.

The proof of Theorem
\hyperlink{thm:twisted-holowinsky-sound}{B}
is
essentially identical to that of Theorem \hyperlink{thm:impl}{A} except
that in the final steps, we are left with sums normalized not
like \eqref{eqn:holowinsky-target-2} but instead like
\begin{equation}\label{eqn:holowinsky-target-3}
  \frac{\sum _{n \leq k} |\lambda(n^2 + 1)|}{k
    \sqrt{
    L(\ad \varphi_k,1)
    L(\ad \varphi_k \otimes \chi_{-4},1)
  }
  }
\end{equation}
(or more precisely their real quadratic analogues involving
$\chi_D$ and polynomials such as $n^2 - D$).  We verify that the
Holowinsky--Soundararajan method successfully applies to such
sums.

\subsection{Plan for this paper}
\S\ref{sec:notat-prel-reduct}--\S\ref{sec:completion-proof} are
devoted to the proof of Theorem \hyperlink{thm:impl}{A}, following the
sketch indicated in \S\ref{sec:main-ideas-proof}.
\S\ref{sec:decay-norm-base} gives the proof of Theorem
\hyperlink{thm:twisted-holowinsky-sound}{B}, borrowing many results from
the previous sections.

\section{Notation and preliminary reductions}\label{sec:notat-prel-reduct}
We adopt the setting of Conjecture
\ref{conj:hque-compact}.  To simplify notation, we drop the
subscripts $k$, thus $\varphi := \varphi_k$.

By the Holowinsky--Soundararajan theorem, we may and shall
assume that $B$ is non-split.  The span of the constant
functions and the Hecke--Maass cusp forms is then dense in the
space of continuous functions on $\mathbf{Y}$ equipped with the
supremum norm, so it suffices to consider the case that $\Psi$
is a Hecke--Maass cusp form.

For any prime $p$, the Hecke operator $T_p$ is defined as
follows.
For a function $f : \mathbf{Y} \rightarrow \mathbb{C}$,
we set
$T_p f(z) := \sum _{\gamma \in \Gamma \backslash R^{(p)}}
f(\gamma z)$, where
$R^{(p)} := \{\gamma \in R : \det(\gamma) = p\}$.

Let $\ram(B)$ denote the set of finite primes at which the
quaternion algebra $B$ does not split.  Since $B$ splits at
$\infty$, we know that $\ram(B)$ is a finite set of even
cardinality.  For $p \in \ram(B)$, the corresponding Hecke
operator $T_p$ on $\mathbf{Y}$ is an involution.  Each such
involution acts on the eigenform $\varphi$ by some sign $\pm 1$, hence leaves
the measure $y^k |\varphi(z)|^2$ invariant.  The operator
$T_p$ on $L^2(\mathbf{Y})$ is self-adjoint and acts on $\Psi$ by
some sign, so the LHS of \eqref{eqn:hque-conj}
vanishes identically unless
\begin{equation}\label{eqn:Psi-tot-even}
  \text{$T_p \Psi = \Psi$ for all $p \in \ram(B)$,}
\end{equation}
as we henceforth assume.

The eigenform $\varphi$ (resp. $\Psi$) generates a cuspidal automorphic
representation $\pi^B$ (resp. $\sigma^B$) of $\PB^\times(\mathbb{A})$.  By the
Jacquet--Langlands lift, we obtain a cuspidal automorphic
representation $\pi$ (resp. $\sigma$) of $\PGL_2(\mathbb{A})$.

Let $p \in \ram(B)$.
The evenness condition
\eqref{eqn:Psi-tot-even} implies that the local component
$\sigma_p^B$ is the trivial representation, hence that
$\sigma_p$ is the Steinberg representation of
$\PGL_2(\mathbb{Q}_p)$ (see \cite[\S56.2]{MR2234120}).  The
local component $\pi_p$ is either the Steinberg representation
or its twist by the nontrivial unramified quadratic character
of $\mathbb{Q}_p^\times$.

For a finite prime $p \notin \ram(B)$, the local components
$\sigma_p = \sigma_p^B$ and $\pi_p = \pi_p^B$ are unramified
principal series representations of $\PGL_2(\mathbb{Q}_p)
\cong B_p/\mathbb{Q}_p^\times$.

We record a special case of Watson's formula \cite[Thm 3]{watson-2008}.
\begin{proposition}\label{prop:watson}
  The squared magnitude of the LHS of \eqref{eqn:hque-conj}
  is equal to
  \begin{equation}\label{eqn:watson}
    c
    \frac{
      \Lambda(\pi \times \pi \times \sigma,1/2)
    }{
      \Lambda(\ad \pi, 1)^2
    },
  \end{equation}
  where $c \geq 0$ depends only upon $\Psi$.
\end{proposition}
Here and henceforth
$\Lambda(\dotsb,s) = L_\infty(\dotsb,s) L(\dotsb,s)$ denotes a
completed $L$-function, including the archimedean local factor
$L_\infty(\dotsb,s)$, while $L(\dotsb,s)$ denotes the finite
part of an $L$-function, given for $\Re(s)$ large enough by a
convergent Euler product $\prod_p L_p(\dotsb,s)$ with $p$
running over the finite primes of $\mathbb{Q}$.  We note that,
e.g., $L(\pi \times \pi \times \sigma, 1/2)$ was denoted
$L(\varphi \times \varphi \times \Psi, 1/2)$ in
\S\ref{sec:main-ideas-proof}.  We denote by (e.g.)
$\eps(\sigma)$ the global $\eps$-factor,
evaluated at the central point $s=1/2$; it factors as the
product of $\eps_p(\sigma) = \eps(\sigma_p)$ over all places $p$ over
$\mathbb{Q}$, finite or infinite.

By the result of Sarnak \cite{Sar01} noted in
\S\ref{sec:introduction}, we may and shall assume that $\pi$ is
non-dihedral.
This assumption is actually a
consequence of our earlier assumption that $\varphi$ is full
level, which forces each local component of $\pi$ at a finite
place to be spherical or an unramified twist of Steinberg.
We prefer to invoke the full level assumption
only when it genuinely simplifies our discussion.

By the factorization
\begin{equation}\label{eqn:factor-triple-product}
  \Lambda(\pi \times \pi \times \sigma,s)
  =  
  \Lambda(\ad \pi \times \sigma,s)
  \Lambda(\sigma,s),
\end{equation}
we see that the RHS of \eqref{eqn:watson}
vanishes unless
$L(\sigma,1/2) \neq 0$,
in which case
\begin{equation}\label{eqn:L-Psi-nonzero}
  \eps(\sigma) = 1.
\end{equation}
For the proof of Theorem \hyperlink{thm:impl}{A}, there is thus
no loss in assuming \eqref{eqn:L-Psi-nonzero}.  However, because
some of the discussion to follow will be used also in the proof
of Theorem \hyperlink{thm:twisted-holowinsky-sound}{B}, we do not impose
\eqref{eqn:L-Psi-nonzero} as a blanket assumption.

By a \emph{nontrivial fundamental discriminant} we mean the discriminant of
a quadratic field extension of $\mathbb{Q}$.
Recall that by class field theory, the following are in natural
bijection:
\begin{itemize}
\item Nontrivial fundamental discriminants $D$.
\item Nontrivial quadratic
characters $\chi_D$ of $\mathbb{A}^\times/\mathbb{Q}^\times$.
\item Quadratic field
extensions $\mathbb{Q} (\sqrt{D})$ of $\mathbb{Q}$.
\end{itemize}
For each nontrivial fundamental discriminant $D$, we write
\begin{itemize}
\item  $\mathcal{O}_D$ for the ring of integers in
  $\mathbb{Q}(\sqrt{D})$,
\item $\mathbb{N}_D$ for the monoid of integral ideals
  in
  $\mathcal{O}_D$,
\item  $\mathfrak{d} = (\sqrt{D})$ for the different ideal, and
\item $\mathcal{N}(\mathfrak{a})$ for the absolute
  norm of a fractional ideal $\mathfrak{a}$ of $\mathcal{O}_D$.
\end{itemize}
We recall that $\mathcal{O}_D$ consists of all elements
$(n + \ell \sqrt{D})/2 \in \mathbb{Q}(\sqrt{D})$ for which
$\ell \in \mathbb{Z}$ and
\begin{equation}\label{eqn:conditions-for-O-D}
  \begin{cases}
    n \in 2 \mathbb{Z}  & \text{ if $D$ is even}, \\
    n \in \ell + 2 \mathbb{Z} & \text{ if $D$ is odd.}
  \end{cases}
\end{equation}
When $D$ is clear from context,
we say that a rational prime $p$ is \emph{split}, \emph{inert}
or \emph{ramified} according to its behavior with respect to
$\mathcal{O}_D$.  We say more generally that a natural number
$n \in \mathbb{N}$ is \emph{split} or \emph{inert} or
\emph{ramified} if it is a product of primes with the indicated
property.  When $D$ is positive, we fix an ordering on the
archimedean places (i.e., real embeddings) $\infty_1, \infty_2$
of $\mathbb{Q}(\sqrt{D})$, with $\infty_1$ the standard
embedding with respect to which $\sqrt{D}$ is positive.

\section{Choice of quadratic character}\label{sec:choice-quadr-field}
In this section we construct a family of quadratic characters
$\chi_D$ relevant for the proof of Theorem \hyperlink{thm:impl}{A}.
\begin{proposition}\label{prop:friedberg-hoffstein}
  Assume \eqref{eqn:L-Psi-nonzero}.
  Then
  there are infinitely many
  nontrivial quadratic characters $\chi_D$ of
  $\mathbb{A}^\times/\mathbb{Q}^\times$
  with the following
  properties:
  \begin{enumerate}[(i)]
  \item $L(\sigma \otimes \chi_D,1/2) \neq 0$.
  \item The archimedean local component $(\chi_D)_{\infty}$ is trivial.
  \item For each $p \in \ram(B)$,
    the local component $(\chi_D)_p$ is the nontrivial unramified quadratic
    character of $\mathbb{Q}_p^\times$.
  \end{enumerate}
\end{proposition}
\begin{proof}
  A result of Friedberg--Hoffstein \cite[Thm B]{MR1343325}
  reduces our task to verifying the existence of at least one
  quadratic character $\chi_D$ satisfying conditions (ii), (iii)
  and for which
  \begin{equation}\label{eqn:eps-sigma-twisted-1}
    \eps(\sigma \otimes \chi_D) = 1.
  \end{equation}
  But our assumption \eqref{eqn:L-Psi-nonzero}
  implies that \eqref{eqn:eps-sigma-twisted-1}
  holds for all such $\chi_D$.
  Indeed, 
  we may factor the sign
  $\eps(\sigma)$
  as a product of local signs
  $\eps(\sigma_p)$,
  with $p$ running over the places of $\mathbb{Q}$
  (finite or infinite),
  and similarly for
  $\eps(\sigma \otimes \chi_D)$.
  We compute for each such $p$ the ratio
  \begin{equation}\label{eqn:}
    \eps(\sigma_p \otimes (\chi_D)_p)/\eps(\sigma_p).
  \end{equation}
  \begin{itemize}
  \item For $p = \infty$,
    the component $(\chi_D)_p$ is trivial,
    so the ratio is $1$.
  \item For $p \in \ram(B)$, the component $(\chi_D)_p$ is the
    nontrivial unramified quadratic character, while $\sigma_p$
    is the Steinberg representation.  By \cite[Prop 3.1.2, Thm
    3.2.2]{Sch02}, the ratio is $-1$.
  \item For finite $p \notin \ram(B)$,
    the component $\sigma_p$ is unramified.
    By \cite[Prop 3.1.1]{Sch02},
    the ratio is $(\chi_D)_p(-1)$.
  \end{itemize}
  Thus
  \begin{equation}\label{eqn:ratio-global-eps}
    \eps(\sigma \otimes \chi_D) / \eps(\sigma)
    =
    (-1)^{\# \ram(B)}
    \prod _{p \notin \{\infty \} \cup \ram(B) }
    (\chi_D)_p(-1).
  \end{equation}
  For $p \in \{\infty \} \cup \ram(B)$, we have
  $(\chi_D)_p(-1) = 1$, so the product over $p$ in
  \eqref{eqn:ratio-global-eps} is $\chi_D(-1) = 1$.  Since
  $\ram(B)$ has even cardinality,
  we have $(-1)^{\# \ram(B)} = 1$.  Thus
  $\eps(\sigma \otimes \chi_D) = \eps(\sigma)$, and so
  \eqref{eqn:eps-sigma-twisted-1} follows from
  \eqref{eqn:L-Psi-nonzero}.
\end{proof}

For $\chi_D$ as in Proposition \ref{prop:friedberg-hoffstein},
our assumptions on the local components imply that
\begin{itemize}
\item $\mathbb{Q}(\sqrt{D})$ is real quadratic, i.e., $D>0$, and
\item each $p \in \ram(B)$ is inert in $\mathcal{O}_D$.
\end{itemize}
We remark that the assumption that $D$ be positive is
unimportant, but helps streamline our discussion; we could just
as well work with negative $D$.

\section{Base change}\label{sec:base-change}
Let $D$ be a nontrivial fundamental discriminant.
Recall that $\pi$ is (assumed) non-dihedral.
By quadratic
base change (see \cite[\S5.3]{MR546613}), we obtain from $\pi$ a
cuspidal automorphic representation $\pi_D$ of
$\PGL_2(\mathbb{A}_{\mathbb{Q}(\sqrt{D})})$.

As in \S\ref{sec:introduction}, let
$\lambda := \lambda_{\varphi} : \mathbb{N} \rightarrow \mathbb{C}$
denote the
multiplicative function
describing the normalized Hecke eigenvalues
of $\pi^B$,
so that
\begin{equation}\label{eqn:}
  L(\pi,s)
  = \sum _{n \in \mathbb{N} }
  \frac{\lambda(n)}{n^s}.
\end{equation}
The normalized Hecke eigenvalues of $\pi_D$ are then
described by the
$\Gal(\mathbb{Q}(\sqrt{D})/\mathbb{Q})$-invariant
multiplicative function
$\lambda_D : \mathbb{N}_D \rightarrow \mathbb{C}$ characterized
by the relation
\begin{equation}\label{eqn:L-pi-D-s-defn}
  L(\pi_D,s)
  :=
  \sum _{\mathfrak{a} \in \mathbb{N}_D}
  \frac{\lambda_D(\mathfrak{a})}{\mathcal{N} (\mathfrak{a})^s}
  =
  L(\pi,s)
  L(\pi \otimes \chi_D,s).
\end{equation}
This relation yields explicit formulas for $\lambda_D$.
\begin{lemma}
For a prime $\mathfrak{p}$ of
$\mathcal{O}_D$ lying over a rational prime $p$, we have
\begin{equation}\label{eqn:lambda-D-prime-powers}
  \lambda_D(\mathfrak{p}^n)
  = \lambda(\mathcal{N}(\mathfrak{p}^n))
\end{equation}
except when $p$ is inert and $p \notin \ram(B)$, in which case
$\lambda_D(\mathfrak{p}^n) = \alpha^{2 n} + \alpha^{2 n - 2} +
\dotsb + \alpha^{-2n}$ for any $\alpha \in \mathbb{C}^\times$
with $\alpha + \alpha^{-1} = \lambda(p)$.
\end{lemma}
\begin{proof}
  The main point is that the local factors $L_p(\pi,s)$
  at primes $p  \in \ram(B)$
  have degree one.
  For convenience, we record the details.
  Let $p$ be a rational prime.  
By
taking Euler factors at $p$ of both sides of
\eqref{eqn:L-pi-D-s-defn},
we obtain
\[
  L_p(\pi_D,s)
  = L_p(\pi,s) L_p(\pi \otimes \chi_D,s).
\]
We view
this identity as one of formal power series in the
variable $X := p^{-s}$.
We argue case-by-case:
\begin{itemize}
\item Suppose that $p$ is split.  Let $\mathfrak{p}$ be any
  prime of $\mathcal{O}_D$ lying over $p$.  Since $\lambda_D$ is
  Galois-invariant, we have
  $L_p(\pi_D,s) = \left( \sum _{n \geq 0}
    \lambda_D(\mathfrak{p}^n) X^n \right)^2$.  On the other
  hand,
  since $p$ is split,
  the local component $(\chi_D)_p$ is trivial,
  and so
  $L_p(\pi,s) = L_p(\pi \otimes \chi_D,s) = \sum_{n \geq 0}
  \lambda(p^n) X^n$.  Since the $\mathbb{C}$-algebra of formal
  power series in $X$ is an integral domain, it follows that
  $\lambda_D(\mathfrak{p}^n) = \pm \lambda(p^n)$ for some sign
  $\pm$ and all $n \geq 0$.  By considering the case $n=0$, we
  deduce that the sign in question is $+$.
  Since $\mathcal{N} (\mathfrak{p}) = p$,
  the required identity follows.
\item Suppose next that $p$ is ramified.  Let $\mathfrak{p}$
  denote the unique prime of $\mathcal{O}_D$ lying over $p$.
  By our local assumptions concerning $\pi$
  and some case analysis,
  we have $L_p(\pi \otimes \chi_D, s) = 1$:
  \begin{itemize}
  \item If $p \notin \ram(B)$,
    then $\pi_p$ is an unramified representation
    while $(\chi_D)_p$ is a ramified character,
    so $(\pi \otimes \chi_D)_p$ is a principal series
    representation induced by a pair of ramified characters.
  \item If $p \in \ram(B)$, then $\pi_p$ is an unramified twist
    of Steinberg, hence $(\pi \otimes \chi_D)_p$ is a ramified
    twist of Steinberg.
  \end{itemize}
  In either case, the $L$-value in question is identically $1$
  (see \cite[\S1.14]{MR3889963}).
  Thus $L_p(\pi_D,s) = L_p(\pi,s)$.  The
  required identity $\lambda_D(\mathfrak{p}^n) = \lambda(p^n)$
  follows.
\item Suppose that $p$ is inert and that $p \in \ram(B)$.  Then
  $\pi_p$ is an unramified twist of Steinberg and $(\chi_D)_p$
  is an unramified character, hence
  $L_p(\pi,s) = (1 - \lambda(p) X)^{-1}$,
  $L_p(\pi \otimes \chi_D,s) = (1 + \lambda(p) X)^{-1}$ and
  $\lambda(p^n) = \lambda(p)^n$.  Therefore
  $L_p(\pi_D,s) = (1 - \lambda(p)^2 X^2)^{-1}$ and so
  $\lambda_D(\mathfrak{p}^n) = \lambda(p^2)^n = \lambda(p^{2 n})
  = \lambda(\mathcal{N} (\mathfrak{p}^n))$.
\item Suppose, finally, that $p$ is inert and that
  $p \notin \ram(B)$.  Let $\alpha$ be as in the statement, so
  that $\{\alpha, \alpha^{-1}\}$ is the multiset of Satake
  parameters for $\pi$ at $p$.
  Then
  $L_p(\pi,s) = (1 - \alpha X)^{-1}(1 - \alpha^{-1} X)^{-1}$,
  while
  $L_p(\pi \otimes \chi_D,s) = (1 + \alpha X)^{-1}(1 +
  \alpha^{-1} X)^{-1}$,
  hence
  \[
    \sum _{n \geq 0 }
    \lambda_D(\mathfrak{p}^n) X^{2 n}
    = L_p(\pi_D,s) = \frac{1}{(1 - \alpha^2 X^2) ( 1 - \alpha^{-2} X^2)}.
  \]
  By expanding the RHS
  as a product of geometric series
  and comparing coefficients of $X^{2 n}$, the required identity follows.
\end{itemize}
\end{proof}

For future reference,
we deduce some consequences of this description.
For a nonzero element $x$ of $\mathcal{O}_D$
we abbreviate $\lambda_D(x) := \lambda_D((x))$.

\begin{lemma}\label{lem:lambda-D-for-ideals}
  Suppose that $\mathfrak{a} \in \mathbb{N}_D$
  is not divisible by any rational prime $p$
  that is either split or inert.
  Let $d_1, d_2 \in \mathbb{N}$,
  with $d_1$ split and $d_2$ inert.
  Then
  $\lambda_D(d_1 d_2 \mathfrak{a})
  =
  \lambda(d_1) \lambda_D(d_2) \lambda(d_1 \mathcal{N}(\mathfrak{a}))$.
\end{lemma}
\begin{proof}
  Our hypothesis implies that we may write $\mathfrak{a}$ as a
  product $\prod_p \mathfrak{p}(p)^{n(p)}$, where
  \begin{itemize}
  \item $p$ runs over non-inert rational primes,
  \item $\mathfrak{p}(p)$ is a prime of $\mathcal{O}_D$ lying
    over $p$, hence of degree one (as is the case for any split
    or ramified prime), and
  \item $n(p) \in \mathbb{Z}_{\geq 0}$.
  \end{itemize}
  To see this, we observe first that for each inert rational
  prime $p$, no prime $\mathfrak{p}$ lying over $p$ divides
  $\mathfrak{a}$.  Otherwise, since $\mathfrak{p}$ is generated
  by $p$, we deduce that $\mathfrak{a}$ is divisible by $p$,
  contrary to hypothesis.  We observe next that if the rational
  prime $p$ splits in $\mathcal{O}_D$ as the product of prime
  ideals $\mathfrak{p}_1 \mathfrak{p}_2$, then at most one of
  the primes $\mathfrak{p}_1$ or $\mathfrak{p}_2$ may divide
  $\mathfrak{a}$, since otherwise we would again deduce that $p$
  divides $\mathfrak{a}$.
  It follows that $\mathfrak{a}$ is divisible by at most one
  element
  of each Galois orbit of primes of $\mathcal{O}_D$.
  The required product decomposition follows.

  Write $d_i = \prod_p p^{m_i(p)}$.
  For $p$ inert, set $n(p) := 0$.
  Then in all cases, $\mathcal{N}(\mathfrak{p}^{n(p)}) = p^{n(p)}$.
  By multiplicativity, the required formula
  reads
  \[
    \prod_p
    \lambda_D(p^{m_1(p) + m_2(p)} \mathfrak{p}(p)^{n(p)})
    =
    \prod_p
    \lambda(p^{m_1(p)})
    \lambda_D(p^{m_2(p)})
    \lambda(p^{m_1(p) + n(p)}).
  \]
  We verify this formula one rational prime $p$ at a time.
  For notational simplicity,
  we abbreviate $m_i := m_i(p)$,
  $\mathfrak{p} := \mathfrak{p}(p)$ and $n := n(p)$.
  \begin{itemize}
  \item Suppose that $p$ is split.
    Then $m_2 = 0$,
    so we must check that
    \[
      \lambda_D(p^{m_1} \mathfrak{p}^{n})
      = \lambda(p^{m_1})
      \lambda(p^{m_1 + n}).
    \]
    Let $\bar{\mathfrak{p}}$ denote the conjugate of
    $\mathfrak{p}$,
    so that $\{\mathfrak{p}, \bar{\mathfrak{p}}\}$ is the set of
    primes
    lying over $p$.
    Then $p^{m_1} \mathfrak{p}^{n}
    = \mathfrak{p}^{m_1 + n} \bar{\mathfrak{p}}^{m_1}$.
    By the multiplicativity of $\lambda_D$,
    it follows that
    $\lambda_D(p^{m_1} \mathfrak{p}^n)
    = \lambda_D(\mathfrak{p}^{m_1 + n})
    \lambda_D(\bar{\mathfrak{p}}^{m_1})$.
    The required identity then follows
    from \eqref{eqn:lambda-D-prime-powers}.
  \item
    If $p$ is inert,
    then $m_1 = n = 0$,
    so the required identity
    is the tautology
    $\lambda_D(p^{m_2})
      = \lambda_D(p^{m_2})$.
  \item Suppose that $p$ is ramified.
    Then $m_1 = m_2 = 0$,
    so we must check that
    $\lambda_D(\mathfrak{p}^n) = \lambda(p^n)$,
    which again follows from
  \eqref{eqn:lambda-D-prime-powers}.
  \end{itemize}
\end{proof}

\begin{lemma}\label{lem:explicate-lambda-D-principal}
  Let $x = (n + \ell \sqrt{D})/2$ be a nonzero
  element of $\mathcal{O}_D$.
  Let $d_1$ (resp. $d_2$)
  denote the largest split (resp. inert) natural number
  dividing $x$.
  Then
  \begin{equation}\label{eqn:lambda-D-via-lambda}
    \lambda_D(x)
    =
    \lambda(d_1)
    \lambda_D(d_2)
    \lambda \left(\frac{|n^2 - D \ell^2|}{4 d_1 d_2^2} \right).
  \end{equation}
\end{lemma}
\begin{proof}
  We apply
  lemma \ref{lem:lambda-D-for-ideals}
  to $\mathfrak{a} := (x/d_1 d_2)$.
\end{proof}

Assume for the remainder of this section that $D$ is positive,
so that $\mathbb{Q}(\sqrt{D})$
splits at $\infty$.
Each archimedean local component $(\pi_D)_{\infty_1},
(\pi_D)_{\infty_2}$
of $\pi_D$
is then isomorphic to the archimedean local component $\pi_\infty$
of $\pi$,
which is the holomorphic discrete series representation
of $\PGL_2(\mathbb{R})$
with weights
\begin{equation}\label{eqn:weights-k-etc}
    \{\dotsc, -k-2, -k, k, k+2, k+4, \dotsc \},
\end{equation}
each occurring with multiplicity one.  In particular,
\begin{equation}\label{eqn:L-infinity-ad-pi-D}
  L_\infty(\ad(\pi_D),s)
  = L_\infty(\ad(\pi),s)^2.
\end{equation}

The pure tensors $\varphi_D \in \pi_D$ for which
\begin{itemize}
\item the local
  component $(\varphi_D)_\mathfrak{p}$ at each finite place
  $\mathfrak{p}$ of $\mathbb{Q}(\sqrt{D})$
  is a newvector \cite{MR0337789}, and
\item the archimedean component $(\varphi_D)_\infty$ has weight
  $(k,-k)$
\end{itemize}
span a one-dimensional
space.
We normalize a specific element $\varphi_D$ of this space,
as follows.
We will use the notation
\begin{equation}\label{eqn:}
  n(x) := \begin{pmatrix}
    1 & x \\
    & 1
  \end{pmatrix},
  \quad
  a(y) := \begin{pmatrix}
    y &  \\
    & 1
  \end{pmatrix}
\end{equation}
to describe elements of $\PGL_2$ over a ring.  Let $\psi_D$
denote the nontrivial unitary character of
$\mathbb{A}_{\mathbb{Q}(\sqrt{D})}/\mathbb{Q}(\sqrt{D})$ whose
infinite component $(\psi_D)_\infty$ is given by
$(x_1,x_2) \mapsto e(x_1 + x_2)$.
For each finite prime $\mathfrak{p}$ of $\mathbb{Q}(\sqrt{D})$,
the local component $(\psi_D)_\mathfrak{p}$
is then trivial on the local component
$\mathfrak{d}_\mathfrak{p}^{-1}$
of the inverse different, but not on any larger fractional
$(\mathcal{O}_D)_\mathfrak{p}$-ideal.
We have the
$\psi_D$-Whittaker expansion
$\varphi_D(g) = \sum _{\xi \in \mathbb{Q}(\sqrt{D})^\times } W
(a(\xi) g)$, where $W$ is the $\psi_D$-Whittaker function given
by the integral
$W(g) = \int_{x \in \mathbb{A}_{\mathbb{Q}(\sqrt{D})} /
  \mathbb{Q}(\sqrt{D})} \varphi_D (n(x) g) \psi_D(-x) \, d x$
with respect to the probability Haar $d x$.  Since $\varphi_D$
is a pure tensor, we may factor $W(g)$ as a product
$\prod_{\mathfrak{p}} W_\mathfrak{p}(g_\mathfrak{p})$ over all
places $\mathfrak{p}$ of $\mathbb{Q}(\sqrt{D})$.
We normalize
$\varphi_D$ in terms of its Whittaker function
$W$, as follows:
\begin{itemize}
\item Let
  $y = (y_\mathfrak{p})_{\mathfrak{p} < \infty}$
  be an element of the finite
  ideles of $\mathbb{Q}(\sqrt{D})$ with corresponding fractional
  ideal $\mathfrak{a}$.
  Thus
  the local component $\mathfrak{a}_\mathfrak{p}$
  is the fractional $(\mathcal{O}_D)_\mathfrak{p}$-ideal
  generated by
  $y_\mathfrak{p}$.
  We
  require that
  $\prod_\mathfrak{p} W_\mathfrak{p}(a(y_\mathfrak{p})) = 0$
  unless the $\mathfrak{d} \mathfrak{a}$ is an integral ideal, in
  which case
  \begin{equation}\label{eqn:}
    \prod_\mathfrak{p} W_\mathfrak{p}(a(y_\mathfrak{p})) =
    \lambda_D(\mathfrak{d} \mathfrak{a}) / \mathcal{N}(\mathfrak{d}
    \mathfrak{a})^{1/2},
  \end{equation}
  where $\lambda := \lambda_{\varphi} : \mathbb{N} \rightarrow
  \mathbb{C}$
  describes the normalized Hecke eigenvalues
  of $\varphi$ as in \S\ref{sec:introduction}.
  For this normalization to make sense,
  we need to know
  the local
  $(\psi_D)_{\mathfrak{p}}$-Whittaker function
  $W_\mathfrak{p}$
  described in the
  $(\psi_D)_{\mathfrak{p}}$-Kirillov model by
  \begin{equation}\label{eqn:local-kirillov-W-frak-p}
    y_\mathfrak{p} \mapsto
    \begin{cases}
      \lambda_D(\mathfrak{d}_\mathfrak{p}
      \mathfrak{a}_\mathfrak{p})
      / \mathcal{N}(\mathfrak{d}_\mathfrak{p}
      \mathfrak{a}_\mathfrak{p})^{1/2}
      & \text{ if } \mathfrak{d}_\mathfrak{p}
      \mathfrak{a}_\mathfrak{p}
      \text{ is integral}, \\
      0 & \text{ otherwise}
    \end{cases}
  \end{equation}
  is a newvector.  By translating by a generator of
  $\mathfrak{d}_\mathfrak{p}$, we may assume that
  $\mathfrak{d}_\mathfrak{p} = 1$ and
  $(\psi_D)_\mathfrak{p}$ is  unramified.  We observe then that
  \eqref{eqn:local-kirillov-W-frak-p} defines an
  $(\mathcal{O}_D)_\mathfrak{p}^\times$-invariant function on
  $\mathbb{Q}(\sqrt{D})_\mathfrak{p}^\times$ whose Mellin
  transform is the local Euler factor
  $L_\mathfrak{p}(\pi_D,s) = \sum_{k \geq 0}
  \lambda_D(\mathfrak{p}^k)/ \mathcal{N}(\mathfrak{p})^{k s}$ of
  \eqref{eqn:L-pi-D-s-defn}.
  We then appeal to the fact  \cite{MR620708}
  that when
  $(\psi_D)_\mathfrak{p}$ is  unramified,
  any vector with these properties is a newvector.

\item For an element $y = (y_1,y_2)$
  of $\mathbb{Q}(\sqrt{D})_\infty^\times
  = \mathbb{R}^\times \times \mathbb{R}^\times$,
  we require that
  \begin{equation}\label{eqn:}
    \prod_{j=1,2} W_{\infty_j}(a(y_j))
    = W_k(y_1) W_k(-y_2),
  \end{equation}
  where
  $W_k : \mathbb{R}^\times \rightarrow \mathbb{C}$ denotes the
  $L^2$-normalized Whittaker function
  \begin{equation}\label{eqn:defn-W-k}
    W_k(y) := 1_{y>0} \Gamma(k)^{-1/2} (4 \pi y)^{k/2} e^{-2 \pi   y}.
  \end{equation}
  For this normalization to make sense,
  we need to know that $y \mapsto W_k(\pm y)$
  defines a vector of weight $\pm k$
  in the holomorphic discrete series representation
  of $\PGL_2(\mathbb{R})$
  with weights \eqref{eqn:weights-k-etc}.
  For this fact,
  we refer to \cite[(80), (82)]{MR3889963}.
\end{itemize}

By restricting $\varphi_D$ to the identity component of
$\PGL_2(\mathbb{A}_{ \mathbb{Q}(\sqrt{D})})$, we obtain a
function on $\PGL_2(\mathbb{R})^+ \times \PGL_2(\mathbb{R})^+$ that identifies with a
Hilbert modular form of weight $(k,-k)$.
We denote
that Hilbert modular form again by
$\varphi_D : \mathbb{H} \times \mathbb{H} \rightarrow
\mathbb{C}$.
Explicitly,
for $z_j = x_j + i y_j$
we set $\varphi_D(z_1,z_2) := \varphi_D(g)$,
where $g_\mathfrak{p} = 1$ for finite $\mathfrak{p}$
and $g_{\infty_j}
= n(x_j) a(y_j)$.
The Fourier expansion of this Hilbert modular
form reads
\begin{equation}\label{eqn:varphi-D-fourier}
  \varphi_D(z_1,z_2)
  =
  \sum _{0 \neq m \in \mathfrak{d}^{-1}}
  e(m_1 x_1 + m_2 x_2)
  W_k(m_1 y_1)
  W_k(-m_2 y_2)
  \frac{\lambda_{D}(m \mathfrak{d})}{
    \mathcal{N} (m \mathfrak{d})^{1/2}
  },
\end{equation}
where
\begin{itemize}
\item $\mathfrak{d}^{-1}$
  is the inverse different,
  i.e.,
  the fractional ideal
  $(1/\sqrt{D})$, and
\item $m_1, m_2$ denote the images of $m$ under
  the real embeddings of $\mathbb{Q}(\sqrt{D})$.
\end{itemize}

We define the $L^2$-norm $\|\varphi_D\|$ by integrating over
$\PGL_2(\mathbb{Q}(\sqrt{D})) \backslash
\PGL_2(\mathbb{A}_{\mathbb{Q}(\sqrt{D})})$ with respect to a
Haar measure.
\begin{lemma}\label{lem:norm-varphi-D}
  \begin{equation}\label{eqn:norm-varphi-D}
    \|\varphi_D\|^2
    = c L(\ad(\pi_D), 1),
  \end{equation}
  where $c > 0$ depends only upon $B$, $D$ and the choice of
  Haar measure.
\end{lemma}
\begin{proof}
  This follows from the normalization
  $\int_{y \in \mathbb{R}^\times} |W_k(y)|^2 \, \frac{d y}{|y|} = 1$
  and a standard formula obtained via Rankin--Selberg theory,
  see for instance \cite[\S3.2.2]{nelson-variance-II} or
  \cite[Lem 2.2.3]{michel-2009}.
\end{proof}

\section{Twisted triple products}\label{sec:twist-triple-prod}
We first choose an element $\Psi '$ of the Jacquet--Langlands
lift $\sigma$ of $\sigma^B$.  Recall that for a finite place $p$
of $\mathbb{Q}$, the local component $\sigma_p$ is unramified
for $p \notin \ram(B)$ and is the Steinberg representation for
$p \in \ram(B)$.  We fix a nonzero pure tensor
$\Psi ' \in \sigma$ whose local component at each finite place
is a newvector and whose archimedean component has weight $0$.
It is invariant for each prime $p$ by the action of the unit
group of the order $\begin{pmatrix}
  \mathbb{Z}_p & \mathbb{Z}_p \\
  d_B \mathbb{Z}_p & \mathbb{Z}_p
\end{pmatrix}$, where
\begin{equation}\label{eq:}
  d_B := \prod_{p \in \ram(B)} p
\end{equation}
denotes
the reduced discriminant of $B$.  We write also
$\Psi ' : \mathbb{H} \rightarrow \mathbb{C}$ for the cuspidal
Hecke--Maass eigenform on $\Gamma_0(d_B) \backslash \mathbb{H}$
given by $\Psi'(x + i y) := \Psi'(g)$ with $g_p = 1$ for finite
$p$ and $g_\infty = n(x) a(y)$.  The Fourier expansion of
$\Psi'$ may be written
\begin{equation}\label{eqn:Psi-prime-fourier}
  \Psi '(x+iy)
  = \sum _{0 \neq \ell \in \mathbb{Z} }
  \frac{\rho(|\ell|)}{|\ell|^{1/2}}
  W_{\Psi}(\ell y) e(\ell x),
\end{equation}
where $\rho : \mathbb{N} \rightarrow \mathbb{C}$
denotes the normalized Hecke eigenvalue
and $W_\Psi : \mathbb{R}^\times \rightarrow \mathbb{C}$
  the Whittaker function.
  For mild convenience, we may and shall
  assume that $\Psi '$ is real-valued.
  We write $\|\Psi '\|$ for the $L^2$-norm,
  defined with respect to some Haar measure.

  Let $D$ be a nontrivial fundamental discriminant.
  Let $\varphi_D$
  be as in \S\ref{sec:base-change}.
We denote by $\res(\varphi_D)$ the restriction of $\varphi_D$ to
$\PGL_2(\mathbb{A})$.  We may form the integral
\[
  \int \res(\varphi_D) \Psi'
\]
taken over
$\PGL_2(\mathbb{Q}) \backslash \PGL_2(\mathbb{A})$ with respect
to some Haar measure.

We record a specialized form of Ichino's twisted triple product
formula \cite{MR2449948}.  The statement involves Asai
$L$-functions and their Rankin--Selberg convolutions.  For a
summary of the relevant properties of these, we refer to
\cite[\S4.1]{MR4193477} and its references.
For our purposes,
what matters
is just
the factorization
\begin{equation}\label{eqn:factor-twisted-asai}
  \Lambda(\asai(\pi_D) \times \sigma,s)
  =
  \Lambda(\ad(\pi) \times \sigma, s)
  \Lambda(\sigma \otimes \chi_D, s).
\end{equation}
\begin{proposition}\label{prop:twisted-watson}
  Retain, as usual,
  the assumptions of \S\ref{sec:notat-prel-reduct}.
  Let $D$ be a positive nontrivial fundamental discriminant.
  There is a finite subset $\mathfrak{c} \subseteq
  \mathbb{R}_{\geq 0}$,
  depending only upon $B, D$ and the choices of Haar measure,
  so that
  \begin{equation}\label{eqn:ichino-twisted-specialized}
    \frac{|\int \res(\varphi_D) \Psi'|^2}{\|\varphi_D\|^2 \|\Psi '\|^2}
    =
    c
    \frac{
      \Lambda(\asai(\pi_D) \times \sigma,1/2)
    }{
      \Lambda(\ad (\pi_D), 1)
    }
  \end{equation}
  for some $c \in \mathfrak{c}$.
  If $D$ satisfies
  the local conditions (ii) and (iii)
  of Proposition \ref{prop:friedberg-hoffstein},
  then $c > 0$.
\end{proposition}
\begin{proof}
  Ichino's formula tells us that
  \eqref{eqn:ichino-twisted-specialized}
  holds
  with $c$ the multiple by a nonzero
  constant of a (finite) product $\prod_p I_p$ over all places $p$
  of $\mathbb{Q}$ of normalized local integrals $I_p$.  To describe
  the local integrals, we fix unitary factorizations
  $\pi_D = \otimes_p (\pi_D)_p$ and $\sigma = \otimes_p \sigma_p$.
  Thus $(\pi_D)_p$ is the tensor product of the components
  $(\pi_D)_\mathfrak{p}$ taken over all places $\mathfrak{p}$ of
  $\mathbb{Q}(\sqrt{D})$ lying over $p$.
  We obtain corresponding factorizations
  $\varphi_D = \otimes_p (\varphi_D)_p$ and
  $\Psi' = \otimes_p \Psi'_p$
  of our vectors.
  Then
  \begin{equation}\label{eqn:}
    I_p
    =
    L_p^{-1}
    \int_{g \in \PGL_2(\mathbb{Q}_p)}
    \frac{
      \langle g (\varphi_D)_p, (\varphi_D)_p \rangle
      \langle g \Psi'_p, \Psi'_p \rangle
    }
    {
      \langle (\varphi_D)_p, (\varphi_D)_p \rangle
      \langle \Psi'_p, \Psi'_p \rangle
    }
    \, d g,
  \end{equation}
  where
  \begin{equation}\label{eqn:}
    L_p
    =
    \frac{\zeta_{\mathbb{Q}(\sqrt{D})_p}(2)
      \zeta_{\mathbb{Q}_p}(2)}{\zeta_{\mathbb{Q}_p}(2)}
    \frac
    {
      L_p(\asai(\pi_D) \times \sigma,1/2)
    }
    {
      L_p(\ad(\pi_D), 1)
      L_p(\ad(\sigma), 1)
    }.
  \end{equation}
  For precise normalizations we refer to \cite{MR2449948}.
  These local integrals have all already been computed in the
  literature, so our task is just to assemble the relevant
  computations.  We rely primarily upon the works of Chen--Cheng
  \cite{MR4033818} and Cheng
  \cite{MR4193477}.

  It is shown in \cite[Prop 6.14]{MR4193477}
  that $I_\infty$ is a nonzero constant.
  (Alternatively, we
  may reduce to Watson's calculations and the comparison of
  local integrals proved in \cite{MR2449948} by noting that
  $I_\infty$ is the same local integral that appears in the
  setting of Proposition \ref{prop:watson}.)

  Let $p$ be a finite prime not in $\ram(B)$.  Then $(\pi_D)_p$
  and $\sigma_p$ are unramified.
  By appeal to
  \begin{itemize}
  \item \cite[Lem 2.2]{MR2449948} or \cite[Prop 4.5, part
    (1)]{MR4033818} if $\mathbb{Q}(\sqrt{D})_p$ is
    the split extension $\mathbb{Q}_p \times \mathbb{Q}_p$ or an
    unramified quadratic extension, and
  \item\cite[Prop 4.7, part
  (1)]{MR4033818} if it is a ramified quadratic
  extension, 
\end{itemize}
we see that $I_p = 1$.

  Let $p \in \ram(B)$.  There are only finitely many possibilities for the splitting behavior of $p$ and for the isomorphism classes of $\pi_p$ and $\sigma_p$, hence only
  finitely many possibilities for $I_p$.  Suppose now that $D$
  satisfies the indicated local conditions.  Then $p$ is inert
  in $\mathcal{O}_D$.  In particular, there is a unique prime
  $\mathfrak{p}$ lying over $p$.  We have noted already that
  $\sigma_p$ is the Steinberg representation and that $\pi_p$ is
  the twist of the Steinberg representation by some (possibly
  trivial) unramified quadratic character $\eta$ of
  $\mathbb{Q}_p^\times$.  The local component
  $(\pi_D)_p = (\pi_D)_\mathfrak{p}$ is the local base change of
  $\pi_p$.  Since
  $\mathbb{Q}(\sqrt{D})_\mathfrak{p}/\mathbb{Q}_p$ is
  unramified, the character $\eta$ restricts trivially to the
  image $\mathbb{Q}_p^{\times 2}$ of the norm map from
  $\mathbb{Q}(\sqrt{D})_\mathfrak{p}^\times$.  By \cite[Prop 2
  (b)]{MR546613}, we deduce that $(\pi_D)_\mathfrak{p}$ is the
  (untwisted) Steinberg representation of
  $\PGL_2(\mathbb{Q}(\sqrt{D})_\mathfrak{p}$.  Under these
  conditions, an exact formula for $I_p$ is given by \cite[Prop
  4.8, part (2)]{MR4033818}, confirming in particular
  that $I_p \neq 0$.
  Thus \eqref{eqn:ichino-twisted-specialized} holds with $c > 0$.
\end{proof}

\begin{remark}
  The key feature of Proposition \ref{prop:twisted-watson} is
  that for $D$ as in Proposition \ref{prop:friedberg-hoffstein},
  the constant $c$ is nonzero.  This property is an analytic
  incarnation of the existence of a nonzero
  $\PGL_2(\mathbb{A})$-invariant functional on
  $\pi_D \otimes \sigma$.  It relies crucially on our choice of
  $D$, specifically on the local conditions at primes
  $p \in \ram(B)$.
\end{remark}

To summarize,
we record a preliminary result
towards Theorem \hyperlink{thm:impl}{A}.
\begin{theorem}\label{thm:goalpost}
  Assume for some positive nontrivial fundamental discriminant
  $D$ satisfying the conclusions of Proposition
  \ref{prop:friedberg-hoffstein} 
  that
  \begin{equation}\label{eqn:goalpost}
    \frac{\int \res(\varphi_D) \Psi '}{L(\ad \pi,1)}
    \rightarrow  0
  \end{equation}
  as $k \rightarrow \infty$.
  Then the conclusion of Conjecture \ref{conj:hque-compact}
  holds.
\end{theorem}
\begin{proof}
  By Watson's triple product formula (Proposition \ref{prop:watson}), we reduce
  to estimating the ratio \eqref{eqn:watson} involving triple
  product $L$-functions.
  By comparing the factorizations
  \eqref{eqn:factor-twisted-asai} and
  \eqref{eqn:factor-triple-product},
  we see that
  \begin{equation}\label{eqn:relate-triple-to-twistd}
    \Lambda(\pi \times \pi  \times \sigma,s)
    =
    \frac{\Lambda(\sigma,s)}{\Lambda(\sigma \otimes \chi_D, s) }
    \Lambda(\asai(\pi_D) \times \sigma,s).
  \end{equation}
  By \eqref{eqn:relate-triple-to-twistd}
  and the
  nonvanishing of
  $L(\sigma \otimes \chi_D,1/2)$,
  we reduce to estimating the ratio
  on the RHS of \eqref{eqn:ichino-twisted-specialized}
  involving twisted triple product $L$-functions.
  By applying the twisted triple product formula
  \eqref{eqn:ichino-twisted-specialized}
  in reverse,
  together with the formulas
  \eqref{eqn:norm-varphi-D} for $\|\varphi_D\|^2$ and
  \eqref{eqn:L-infinity-ad-pi-D} for $L_\infty(\ad(\pi_D),1)^2$,
  we reduce to estimating the LHS of \eqref{eqn:goalpost}.
\end{proof}

\begin{remark}
\label{rmk:summary-thus-far}
Let us summarize what we have done thus far.  We
started with automorphic representations $\pi^B$ and $\sigma^B$
of the non-split group $\PB^\times(\mathbb{A})$.  Our aim was to
estimate certain $\PB^\times(\mathbb{A})$-invariant period
integrals defined on $\pi^B \otimes \pi^B \otimes \sigma^B$ in
terms of other period integrals taken over split quotients.  The
triple product period on $\PGL_2$ was not a suitable candidate
because, in view of Prasad's uniqueness theorem \cite[Thm
1.2]{MR1059954}, $\pi \otimes \pi \otimes \sigma$ is not
\emph{locally distinguished}: it does not admit a nonzero
$\PGL_2(\mathbb{A})$-invariant trilinear functional.  The
twisted triple product periods for $\pi_D \otimes \sigma$ were
natural candidates: the corresponding $L$-values share a degree
six Euler factor with those attached to our original periods,
and the freedom to vary $D$ allowed us to arrange that
$\pi_D \otimes \sigma$ be locally distinguished.
The local distinction of $\pi_D \otimes \sigma$ entered into our
analysis implicitly, through the proof of Proposition
\ref{prop:twisted-watson}; it may be verified alternatively
using the $\eps$-factor criterion of \cite[Thm D]{MR1198305}
(see also \cite[\S5]{MR2369492}, \cite{MR2438071}) and explicit
calculations as in the proof of Proposition
\ref{prop:friedberg-hoffstein}.\footnote{
  We sketch the details here.
The cited criterion says that $\pi_D \otimes \sigma$ is locally
distinguished if and only if for each place $p$ of $\mathbb{Q}$,
one has $\eps_p(\pi_D \otimes \sigma) = (\chi_D)_p(-1)$;
moreover, $\pi_D \otimes \sigma$ fails to be locally
distinguished at $p$ only if the local components of $\pi_D$ and
$\sigma$ belong to the discrete series.  At $p = \infty$, the
local component of $\sigma$ belongs to the principal series, so
local distinction holds.  At a finite prime $p \notin \ram(B)$,
the local components of $\pi$ and $\sigma$ are unramified, and
we compute using \cite[Prop 3.1.1]{Sch02} that
$\eps_p(\pi_D \otimes \sigma) = \eps_p(\sigma \otimes \chi_D) =
(\chi_D)_p(-1)$.  At a finite prime $p \in \ram(B)$, the local
component of $\chi_D$ is the unramified quadratic character,
thus $(\chi_D)_p(-1) = 1$; using that $\pi_p$ is an unramified
twist of Steinberg and that $\sigma_p$ is (untwisted) Steinberg,
we compute
as in the
proof of Proposition
\ref{prop:twisted-watson}
that
$\eps_p(\pi_D \otimes \sigma) = \eps_p(\sigma \otimes \chi_D) =
1$.}
\end{remark}

\section{Approximate functional equation for
  periods}\label{sec:appr-funct-equat}
Let $D$ be a positive nontrivial fundamental discriminant.
We aim to
evaluate
the integrals
$\int \res(\varphi_D) \Psi '$
in terms of the Fourier coefficients of $\varphi_D$ and $\Psi'$.

It will be convenient first to rewrite those integrals
classically.  The function
$\res(\varphi_D) \Psi ' : \PGL_2(\mathbb{Q}) \backslash
\PGL_2(\mathbb{A}) \rightarrow \mathbb{C}$ is right-invariant
under the action of $\SO(2)$ at the infinite place and, for each
prime $p$, the action of the unit group of
the order
$\begin{pmatrix}
  \mathbb{Z}_p & \mathbb{Z}_p \\
  d_B \mathbb{Z}_p & \mathbb{Z}_p
\end{pmatrix}$.  By strong approximation as in \cite[Lem
B.1]{MR3356036}, we deduce that with suitable
normalization of Haar measure,
\begin{equation}\label{eqn:}
  \int  \res(\varphi_D) \Psi'
  =
  \int_{z \in \Gamma_0(d_B) \backslash \mathbb{H}}
  \varphi_D(z,z)
  \Psi '(z) \, \frac{d x \, d y}{y^2}.
\end{equation}

We would like to evaluate (or at least estimate) such integrals
in terms of the Fourier coefficients of $\varphi_D$ and of
$\Psi '$.  To address this problem, we might be tempted to apply
Holowinsky's ``$Y$-thickening technique'' \cite[\S3.1]{MR2680498}.
Unfortunately, to apply that technique effectively here seems to
require more \emph{a priori} control over $\res(\varphi_D)$ than
is available.  For instance, to estimate the analogue of the
quantities ``$R_\varphi(Y)$'' considered in \cite[Lem
3.1a]{MR2680498} seems to require a sharp bound for the
$L^1$-norm of $\res(\varphi_D)$, which seems difficult to
achieve.  By contrast, for the split analogue of our discussion
($D$ is a square,
$\mathbb{Q}(\sqrt{D}) = \mathbb{Q} \times \mathbb{Q}$, and
$\varphi_D = \varphi \otimes \bar{\varphi}$), the $L^1$-norm of
$\res(\varphi_D)$ is simply the squared $L^2$-norm of $\varphi$,
which we may normalize to be $1$.  We instead appeal to the
following ``approximate functional equation'' for integrals of
automorphic forms.
\begin{proposition}\label{prop:AFE}
  Let $\Gamma$ be a finite index subgroup of $\SL_2(\mathbb{Z})$
  with $-1 \in \Gamma$.
  Let $\phi : \Gamma \backslash \mathbb{H} \rightarrow
  \mathbb{C}$
  be a bounded continuous function
  satisfying $|\phi(\tau z)| \leq C y^{-\alpha}$
  for some fixed $C, \alpha > 0$,
  all $z = x + i y$ with $y \geq 1$,
  and all $\tau \in \SL_2(\mathbb{Z})$.
  Set
  \begin{equation}\label{eqn:}
    a_0(y) :=
    \int_{x \in \mathbb{R}/\mathbb{Z}}
    \sum _{\tau \in \Gamma \backslash
      \SL_2(\mathbb{Z})}
    \phi (\tau (x + i y))
    \, d x
  \end{equation}
  and, for $\Re(s) > 1$,
  \begin{equation}\label{eqn:}
    \tilde{a}_0(s)
    =
    \int_{y \in \mathbb{R}^\times_+} a_0(y)  y^{s} \, \frac{d y}{y^2}.
  \end{equation}
  Then for $\delta > 0$,
  \begin{equation}\label{eqn:approx-func-eqn-for-periods}
    \int_{\Gamma \backslash \mathbb{H}}
    \phi(z) \, \frac{d x \, d y}{y^2}
    =
    \int_{\Re(s) = 1+\delta}
    (2 s - 1) 2 \xi(2 s)
    \tilde{a}_0(s)
    \, \frac{d s}{2 \pi i},
  \end{equation}
  where $\xi(s) := \Gamma_{\mathbb{R}}(s) \zeta(s)$,
  $\Gamma_{\mathbb{R}}(s) = \pi^{-s/2} \Gamma(s/2)$
  denotes the completed Riemann zeta function.
\end{proposition}
\begin{proof}
  This is the special case $H(s) = s$ of \cite[Thm
  5.6]{MR3356036} (corrected by requiring that
  $-1 \in \Gamma$).
\end{proof}
\begin{remark}
  We refer to \cite[\S5]{MR3356036} for some discussion
  (motivated by numerical applications, but relevant for
  analytic ones) of the relationship between ``$Y$-thickening''
  and Proposition \ref{prop:AFE}, and to
  \cite{2018arXiv180209740C} and \cite[\S4]{MR3889557} for
  further applications.
\end{remark}

We apply this result to $\Gamma = \Gamma_0(d_B)$ and
$\phi(z) := \varphi_D(z,z) {\Psi '(z)}$.
The main point in evaluating
$a_0(y)$
is then the calculation
\begin{equation}\label{eqn:expand-constant-term-phi}
  \int_{x \in \mathbb{R}/\mathbb{Z}} \phi(x + i y) \, d x
  =
  \sum _{
    \substack{
      0 \neq \ell \in \mathbb{Z}, \\
      0 \neq m \in \mathfrak{d}^{-1} : \\
      \tr(m)
      = \ell
    }
  }
  \frac{\rho(|\ell|)}{|\ell|^{1/2}}
  \frac{\lambda_{D}(m \mathfrak{d})}{
    \mathcal{N} (m \mathfrak{d})^{1/2}
  }
  W_{\Psi}(\ell y)
  W_k(m_1 y)
  W_k(-m_2 y),
\end{equation}
which follows by opening the Fourier series
\eqref{eqn:varphi-D-fourier}, \eqref{eqn:Psi-prime-fourier} and
using that $\Psi'$ is real-valued.
By combining this with
similar calculations at the other cusps of $\Gamma_0(d_B)$, we
will verify the following.
\begin{proposition}
  We have
  \begin{equation}\label{eqn:twisted-period-via-sums-1}
    \int \res(\varphi_D) \Psi '
    =
    \sum _{
      \substack{
        0 \neq \ell \in \mathbb{Z}, \\
        0 \neq m \in \mathfrak{d}^{-1} : \\
        \tr(m)
        = \ell
      }
    }
    \frac{\rho(|\ell|)}{|\ell|^{1/2}}
    \frac{\lambda_{D}(m \mathfrak{d})}{
      \mathcal{N} (m \mathfrak{d})^{1/2}
    }
    V_k(\ell, m)
  \end{equation}
  where
  \begin{equation}\label{eqn:}
    V_k(\ell,m)
    :=
    \int_{y \in \mathbb{R}^\times_+}
    h(y)
    W_{\Psi}(\ell y)
    W_k(m_1 y)
    W_k(-m_2 y)
    \, \frac{d y}{y^2}
  \end{equation}
  with
  $h \in C^\infty(\mathbb{R}^\times_+)$
  defined by the rapidly-convergent
  Mellin integral
  \begin{equation}\label{eqn:h-of-y-defn-via-contour-integral}
    h(y) :=
    \int _{\Re(s) = 1+\delta } (2 s- 1) 2 \xi(2 s)
    \left(\sum _{d | d_B } d^s\right)
    y^s \, \frac{d s}{2 \pi i }.
  \end{equation}
\end{proposition}
\begin{proof}
  For
  the quotient
  $\Gamma_0(d_B) \backslash \SL_2(\mathbb{Z})$,
  we take the coset representatives
  $w(d) n(j)$,
  where $d$ traverses the set of positive divisors of $d_B$,
  $j$ runs over $\mathbb{Z}/d$,
  and $w(d), n(j) \in \SL_2(\mathbb{Z})$
  are described by
  \begin{equation}\label{eqn:}
    w(d) \equiv \begin{pmatrix}
      0 & 1 \\
      -1 & 0
    \end{pmatrix}
    \mod{d},
    \quad 
    w(d) \equiv \begin{pmatrix}
      1 & 0 \\
      0 & 1
    \end{pmatrix}
    \mod{d_B/d},
    \quad
    n(j)
    :=
    \begin{pmatrix}
      1 & j \\
      0 & 1
    \end{pmatrix}.
  \end{equation}
  Let $p$ be a prime divisor of $d_B$.  Then $p$ remains prime
  in $\mathcal{O}_D$, and the local components of $\pi_D$ and
  $\sigma$ at $p$ are Steinberg representations, hence have
  local Atkin--Lehner eigenvalue $-1$ (see \cite[Prop
  3.1.2]{Sch02}).
  It follows that
  \begin{equation}\label{eqn:}
    \varphi_D(w(d) z, w(d) z)
    = \mu(d) \varphi_D(z/d, z/d),
    \quad 
    \Psi '(w(d) z) = \mu(d) \Psi '(z/d),
  \end{equation}
  where $\mu$ denotes the M{\"o}bius function.
  Using that $\mu(d)^2 = 1$
  and $n(j)(x + i y) = x + j + i y$,  we deduce that
  \begin{equation}\label{eqn:}
    a_0(y)
    =
    \sum_{d|d_B}
    \int_{x \in \mathbb{R}/\mathbb{Z}}
    \sum_{j \in \mathbb{Z}/d}
    \phi
    \left(\frac{x + j + i y }{d} \right) \, d x
    =
    \sum_{d|d_B}
    d
    \int_{x \in \mathbb{R}/\mathbb{Z}}
    \phi \left(x + \frac{ i y }{d}\right) \, d x,
  \end{equation}
  thus
  \begin{equation}\label{eqn:}
    \tilde{a}_0(s)
    =
    \left(\sum _{d | d_B}
    d^s \right)
    \int_{y \in \mathbb{R}^\times}
    \left(
    \int_{x \in \mathbb{R}/\mathbb{Z}}
    \phi(x + i y) \, d x \right)
    y^s
    \, \frac{d y}{y^2}.
  \end{equation}
  We evaluate the inner integral over
  $x$ as in \eqref{eqn:expand-constant-term-phi}
  and insert the resulting formula
  for $\tilde{a}_0(s)$
  into \eqref{eqn:approx-func-eqn-for-periods},
  giving the formula
  \begin{equation}\label{eq:iterated-sum-int}
    \int \res (\varphi_D) \Psi ' =
    \int_{\Re(s) = 1 + \delta}
    \int_{y \in \mathbb{R}^\times}
    \sum_{m \in \mathfrak{d}^{-1}}
    \mathcal{R}(s)
    \mathcal{T}(m,y)
    y^s
    \, \frac{d y}{y^2}
    \, \frac{d s}{ 2 \pi  i},
  \end{equation}
  where
  \[
    \mathcal{R}(s) := (2 s - 1) 2 \xi(2 s) \left( \sum _{d |
        d_B} d^s \right),
  \]
  \[
    \mathcal{T}(m,y) := 1_{ m \neq 0} 1_{ \ell := \tr(m) \neq 0}
    \frac{\rho(|\ell|)}{|\ell|^{1/2}} \frac{\lambda_{D}(m
      \mathfrak{d})}{ \mathcal{N} (m \mathfrak{d})^{1/2} }
    W_{\Psi}(\ell y) W_k(m_1 y) W_k(-m_2 y).
  \]
  We may shift the contour to $\Re(s) = c$ with $c$ large enough
  but fixed.  We observe the following estimates:
  \begin{itemize}
  \item $\mathcal{R}(s) \ll_{c,A} (1 + |s|)^{-A}$ for each fixed
    $A$, due to the rapid decay coming from the $\Gamma$-factor
    in $\xi(2 s)$.
  \item $|y^s| \leq y^c$.
  \item
    $\mathcal{T}(m,y) \ll_{A,\eps} (1 + |\ell y|)^{-A} (1 + |m_1
    y|)^{-A} (1 + |m_2 y|)^{-A}$ by the Hecke bound for the Hecke
    eigenvalues, trivial bounds for the Whittaker functions near
    zero, and the rapid decay of the Whittaker functions near
    infinity.
  \item     $\sum_{m \in \mathfrak{d}^{-1}} |\mathcal{T}(m,y)|
    \ll_A y^{-2} + y^{-A}$ by the previous estimate
    and the bound $\O(X^2)$
    for the number of $m \in \mathfrak{d}^{-1}$
    with $|m_1|, |m_2| \leq X$.
  \end{itemize}
  Using these estimates,
  we deduce that the three-fold iterated sum/integral on the
  RHS of \eqref{eq:iterated-sum-int} converges absolutely.  We
  may thus rearrange it as $\sum_{m} \int_y \int_s (\dotsb)$.
  Shifting the contour back to $\Re(s) = 1 + \delta$ then yields
  the required formula.
\end{proof}

\section{Asymptotics of archimedean
  integrals}\label{sec:asympt-arch-integr}
We retain the setting of \S\ref{sec:appr-funct-equat}.
Let $\ell$ be a nonzero integer,
and let $m \in \mathfrak{d}^{-1}$
with $\trace(m) = \ell$.
Since
$\mathfrak{d}^{-1}$ is the fractional $\mathcal{O}_D$-ideal
generated by $1/\sqrt{D}$,
we may write
\begin{equation}\label{eqn:}
  m = \frac{\ell + n / \sqrt{D}}{2}  
\end{equation}
for some $n$ satisfying
\eqref{eqn:conditions-for-O-D}.
In this way, we may view
the RHS of \eqref{eqn:twisted-period-via-sums-1}
as a sum over integers $\ell$ and $n$,
with $\ell$ nonzero,
satisfying \eqref{eqn:conditions-for-O-D}.
Since $\infty_1$ is the standard embedding
$\mathbb{Q}(\sqrt{D}) \hookrightarrow \mathbb{R}$
with respect to which $\sqrt{D}$ is positive,
we may assume that $n$ is positive,
since otherwise $V_k(\ell,m)$
vanishes due to the support condition on $W_k$.

We turn now to estimates.  The following conventions concerning
asymptotic notation and terminology will be in effect for the
remainder of the paper.  We say that a quantity is \emph{fixed}
if it is independent of our sequence parameter $k$.
For instance, the eigenform $\Psi$ is fixed.
We let
$\eps > 0$ and $N \in \mathbb{Z}_{\geq 0}$ denote fixed
quantities, with $\eps$ sufficiently small and $N$ sufficiently
large.  We use the notation $A = \O(B)$ or $A \ll B$ to denote
that $|A| \leq C |B|$ for some fixed $C \geq 0$, which we allow
to depend upon any previously mentioned fixed quantities.  In
particular, such implied constants $C$ may depend upon the fixed
quantities $\eps, N$ and $\Psi$.

\begin{proposition}\label{prop:asympt-arch-integr}
  We have
  \begin{equation}\label{eqn:sum-just-after-truncation}
  \int \res(\varphi_D) \Psi '
  =
  \sum _{
    \ell,n : |\ell| < k^\eps
  }^\sharp
  \lambda_{D} \left(\frac{n + \ell \sqrt{D}}{2} \right)
  \frac{f_\ell (n/k)}{k} + \O(k^{-1+\eps}),
\end{equation}
where
\begin{itemize}
\item  the symbol $\sharp$
indicates that $\ell$ and $n$ are integers, with $\ell$ nonzero and $n$
positive,
satisfying
the congruence condition
\eqref{eqn:conditions-for-O-D},
and
\item
  the $f_\ell$ are
  smooth functions on $\mathbb{R}^\times_+$
  satisfying, for $\mathcal{S}_N$ as in
  Definition \ref{defn:sobolev-norms}, the estimates
  \begin{equation}\label{eqn:sobolev-estimates-f-ell}
  \mathcal{S}_N(f_\ell) \ll |\ell|^{-N}.
\end{equation}
\end{itemize}
\end{proposition}

The proof occupies
the remainder of \S\ref{sec:asympt-arch-integr}.

Set
\begin{equation}\label{eqn:}
  h_\ell(y) := h(y) W_\Psi(\ell y) / y.
\end{equation}
By substituting the definition \eqref{eqn:defn-W-k} of $W_k$ and
executing the change of variables
$y \mapsto y \sqrt{D} / 2 \pi n$, we see that
$V_k(\ell,m) = 0$ unless $n^2 - \ell^2 D > 0$,
in which case
\begin{equation}\label{eqn:V-k-nice-expr}
  V_k(\ell,m)
  =
  \frac{(1 - D \ell^2 / n^2)^{k/2}}{\Gamma(k)}
  \int_{y \in \mathbb{R}^\times_+}
  h_\ell \left(\frac{\sqrt{D}}{2 \pi n } y \right)
  y^k e^{-y}
  \, \frac{d y}{y}.
\end{equation}

We observe that the function $h_\ell(y)$ and its derivatives decay
rapidly with respect to both $y$ (tending either to $0$ or
$\infty$) and $\ell$:
\begin{lemma}
  We have
  \begin{equation}\label{eqn:sobolev-estimate-h-ell}
    \mathcal{S}_N(h_\ell) \ll |\ell|^{-N}.
  \end{equation}
\end{lemma}
\begin{proof}
  By shifting contours in the definition
  \eqref{eqn:h-of-y-defn-via-contour-integral} of $h$, we see
  that $h(y) \ll y^{N}$ as $y \rightarrow 0$ and
  $h(y) = c + \O( y^{-N})$ as $y \rightarrow \infty$ for some
  fixed $c > 0$.  On the other hand, the Whittaker function
  $W_\Psi(y)$ is $\ll y^{1/2-7/64} \ll 1$ for small $y$ and decays exponentially
  for large $y$.  By these and similar estimates for
  derivatives, the required conclusion follows.
\end{proof}

Using what amounts to the rapid decay of $y^k e^{-y}$ for large
$k$ near both $0$ and $\infty$, we verify that $V_k(\ell,m)$ is
small unless $n$ is of size $k$:
\begin{lemma}
  If $n \geq k^{1+\eps}$ or $n \leq k^{1-\eps}$,
  then
  \begin{equation}\label{eqn:}
    V_k(\ell,m) \ll (k n |\ell|)^{-N}
  \end{equation}
\end{lemma}
\begin{proof}
  By \eqref{eqn:sobolev-estimate-h-ell},
  we have
  \[
    h_\ell \left(\frac{\sqrt{D}}{2 \pi n } y \right) \ll |\ell|^{-N}
    \min \left( \frac{y}{n}, \frac{n}{y} \right)^{N}.
  \]
  Since
  $(y/n)^{N}
  \leq (n/k)^{-N} (y/k)^N$
  and
  $(n/y)^N \leq (n/k)^N (k/y)^{N}$,
  it follows that
  \[
    V_k(\ell,m)
    \ll
    |\ell|^{-N}
    \left(\frac{n}{k} \right)^{-N}
    \frac{\Gamma(k+N)}{k^N \Gamma(k)}
  \]
  and
  \[
    V_k(\ell,m)
    \ll
    |\ell|^{-N}
    \left(\frac{n}{k} \right)^{N}
    \frac{k^N \Gamma(k-N)}{\Gamma(k)}
  \]
  We apply the first of these estimates when $n \geq k^{1+\eps}$
  and the second when $n \leq k^{1-\eps}$.
  Since $\Gamma(k+N) \ll k^N \Gamma(k)$,
  and $k^N \Gamma(k-N) \ll \Gamma(k)$,
  we may conclude by
  appealing to our hypothesis on $n$ and replacing $N$ with
  something sufficiently large in terms of $N$ and $\eps$.
\end{proof}

These estimates imply already that for each $\ell$, the
contribution to the sum on the RHS of
\eqref{eqn:twisted-period-via-sums-1} from
$n$ outside the interval $(k^{1-\eps}, k^{1+\eps})$
is negligible, i.e.,
of size $\O(k^{-N}  (1 + |\ell|)^{-N})$.
The contribution to the remaining
sum from $|\ell| \geq k^{\eps}$ is likewise negligible.
We are
left with
\begin{equation}\label{eqn:sum-just-after-truncation-0}
  \int \res(\varphi_D) \Psi '
  =
  \sum _{
    \substack{
      \ell,n : \\ |\ell| < k^\eps, k^{1-\eps} < n < k^{1+\eps} , \\
      m := (\ell + n/\sqrt{D})/2
    }
  }^\sharp
  \frac{\rho(|\ell|)}{|\ell|^{1/2}}
  \frac{\lambda_{D}(m \mathfrak{d})}{
    \mathcal{N} (m \mathfrak{d})^{1/2}
  }
  V_k(\ell,m) + \O(k^{-N}),
\end{equation}
with the symbol $\sharp$
as in the statement of Proposition \ref{prop:asympt-arch-integr}.

We may analyze the remaining sum
by expanding $h_\ell$ via Mellin transform
and appealing to asymptotic formulas for $\Gamma(s+k)/\Gamma(s)$,
exactly as in Luo--Sarnak \cite[p877-878]{luo-sarnak-mass}.  For
completeness and variety of presentation, we record an
alternative argument using the following elementary estimate due
to Iwaniec \cite[Lem C]{IwaniecQUENotes}.
\begin{lemma}\label{lem:iwaniec-incomplete-gamma-integral}
  For $f \in C_c^\infty(\mathbb{R}^\times_+)$
  and $s > 0$,
  the integral
  \begin{equation}\label{eqn:}
    J(s)
    :=
    \frac{1}{\Gamma(s)}
    \int_{\mathbb{R}^\times_+}
    f(y)
    y^s e^{- y} \, \frac{d y}{y}
  \end{equation}
  satisfies the estimate
  \begin{equation}\label{eqn:}
    |J(s) - f(s)|
    \leq
    \frac{s}{2}
    \|f''\|_{L^\infty}
  \end{equation}
\end{lemma}
\begin{proof}
  We appeal
  to Taylor's theorem
  with remainder in the form
  \[
    |f(y) - f(s) - (y-s) f'(s)|
    \leq
    \frac{1}{2}
    (y-s)^2
    \|f''\|_{L^\infty}
  \]
  and note that
  \[
    \int_{\mathbb{R}^\times_+}
    (y-s) 
    y^s e^{-y} 
    \, \frac{d y}{y}
    =
    \Gamma(s+1) - s\Gamma(s) = 0
  \]
  and that
  \[
    \int_{\mathbb{R}^\times_+}
    (y-s)^2
    y^s e^{-y} 
    \, \frac{d y}{y}
    =
    \Gamma(s+2) -
    2 s \Gamma(s+1) + s^2 \Gamma(s)
    =
    s \Gamma(s).
  \]
\end{proof}

It follows that
\begin{equation}\label{eqn:}
  V_k(\ell,m)
  =
  \left(1 - \frac{D \ell^2 }{ n^2} \right)^{k/2}
  h_\ell \left(\frac{k \sqrt{D}}{2 \pi n }  \right)
  + \O \left( \frac{k}{n^2}  (1 + |\ell|)^{-N} \right).
\end{equation}

Let $\ell,n,m$
be
as in \eqref{eqn:sum-just-after-truncation-0}.
Then
\begin{equation}\label{eqn:}
  \left(1 - \frac{D \ell^2 }{ n^2} \right)^{k/2} =
  1 + \O( k^{4 \eps - 1}),
\end{equation}
\begin{equation}\label{eqn:}
  \mathcal{N}(m \mathfrak{d})^{1/2}
  =
  \frac{\sqrt{n^2 - \ell^2 D}}{2}  = \frac{n}{2} + \O(k^{-1+\eps}),
\end{equation}
and
\begin{equation}\label{eqn:}
  \lambda_{D}(m \mathfrak{d})
  = \lambda_D \left(\frac{n + \ell \sqrt{D}}{2}\right),
\end{equation}
where for a nonzero element $x$ of $\mathcal{O}_D$
we abbreviate $\lambda_D(x) := \lambda_D((x))$.
For convenience,
we may use
Deligne's results
to bound $\lambda_{D}(m \mathfrak{d})$ 
by the number of divisors of the ideal $m \mathfrak{d}$, which
in turn is $\O( k^\eps)$.
(The averaged form of this bound following
from Rankin--Selberg theory would also
suffice for our purposes.)
It follows that
\begin{equation}\label{eqn:nearly-final-after-asymp-analysis}
  \int \res(\varphi_D) \Psi '
  =
  \sum _{
    \substack{
      \ell,n : \\ |\ell| < k^\eps, k^{1-\eps} < n < k^{1+\eps} , \\
    }
  }^\sharp
  \lambda_{D}
  \left(\frac{n + \ell \sqrt{D}}{2}\right)
  \frac{f_\ell (n/k)}{k} + \O(k^{-1+10 \eps})
\end{equation}
with
\begin{equation}\label{eqn:}
  f_\ell(u)
  :=
  \frac{\rho(|\ell|)}{|\ell|^{1/2}}
  \frac{2}{u}
  h_\ell
  \left(\frac{\sqrt{D}  }{2 \pi u}\right)
\end{equation}
The estimates \eqref{eqn:sobolev-estimates-f-ell} for $f_\ell$ are
satisfied, so
we incur negligible error in
removing the summation condition
$k^{1-\eps} < n < k^{1+\eps}$
and restricting the sum further to $|\ell| < k^{\eps/10}$.
After renaming $\eps$, we obtain the conclusion
of Proposition \ref{prop:asympt-arch-integr}.

\section{Proof of Theorem \protect\hyperlink{thm:impl}{A}}\label{sec:completion-proof}
We assume Conjecture \ref{conj:non-split-sums} and must deduce
Conjecture \ref{conj:hque-compact}.  We retain the asymptotic
notation and terminology of \S\ref{sec:asympt-arch-integr}.  We
may assume the $\eps$-factor condition \eqref{eqn:L-Psi-nonzero}
and may thus find a positive nontrivial fundamental discriminant $D$ satisfying the
conclusion of Proposition \ref{prop:friedberg-hoffstein}.  By
Theorem \ref{thm:goalpost}, we
reduce to verifying
\eqref{eqn:goalpost}.  We use the lower bound
$L(\ad \pi, 1) \gg k^{-\eps}$, known in stronger form by
\cite{HL94}, and the asymptotic formula
\eqref{eqn:sum-just-after-truncation} for
$\int \res(\varphi_D) \Psi '$.  We evaluate $\lambda_D$ using
Lemma \ref{lem:explicate-lambda-D-principal}.  The natural
numbers $d_1,d_2$ defined in that lemma depend only upon the
congruence class $a \pmod{2 \ell}$ of $n$.  Let us write
$d_j = d_j(a)$ to indicate that dependence.  The restrictions on
$n$ implied by the $\sum_n^\sharp$ notation likewise depend only
upon $a$, so let us accordingly write $\sum_a^\sharp$.  We
obtain
\begin{equation}\label{eqn:complete-proof-1}
  \frac{\int \res(\varphi_D) \Psi '}{L(\ad \pi,1)}
  =
  \sum _{0 \neq |\ell| < k^\eps }
  \sum _{a \in \mathbb{Z}/2 \ell}^\sharp
  \lambda(d_1(a)) \lambda_D(d_2(a))
  S(\ell,a)
  + \O(k^{-1+\eps}),
\end{equation}
where
\begin{equation}\label{eqn:}
  S(\ell,a)
  :=
  \frac{
    1
  }{
    k L(\ad \pi,1)
  }
  \sum _{n \equiv a(2\ell)}
  \lambda \left(\frac{|n^2 - D \ell^2|}{4 d_1 d_2^2}\right)
  f_\ell \left( \frac{n}{k} \right).
\end{equation}
By our assumption of Conjecture \ref{conj:non-split-sums}, we
may find some fixed $N \in \mathbb{Z}_{\geq 0}$ so that for each
fixed $\eps > 0$, we have for large enough $k$ the inequality
\begin{equation}\label{eqn:}
  |S(\ell,a)|
  \leq
  \eps |\ell|^{N}
  \mathcal{S}_N(f_\ell).
\end{equation}
Using the trivial estimate
$|\lambda(d_1) \lambda_D(d_2)| \leq 10 |\ell|^{10}$,
we deduce that
the LHS of \eqref{eqn:complete-proof-1}
is bounded in magnitude by
\begin{equation}\label{eqn:complete-proof-2}
  20 \eps
  \sum _{0 \neq |\ell| < k^\eps }
  |\ell|^{N + 11}
  \mathcal{S}_N(f_\ell)
  + \O(k^{-1+\eps}).
\end{equation}
By the estimate \eqref{eqn:sobolev-estimate-h-ell} (applied with
a larger value of $N$), we see that the sum over $\ell$ in
\eqref{eqn:complete-proof-2} is bounded by some fixed quantity
depending only upon $N$.  Taking $\eps$ sufficiently small in
terms of $N$, we conclude that 
\eqref{eqn:complete-proof-2} can be made arbitrarily small.
This completes
the required deduction of \eqref{eqn:goalpost}.

\section{Proof of Theorem
  \protect\hyperlink{thm:twisted-holowinsky-sound}{B}}\label{sec:decay-norm-base}
We adopt the setting of \S\ref{sec:twist-vari-arithm}.  We set
$B := M_2(\mathbb{Q})$, so that the discussion of
\S\ref{sec:notat-prel-reduct} applies.  The set $\ram(B)$ is
empty, so any conditions concerning $p \in \ram(B)$ hold
tautologically.  We drop subscripts as before:
$\varphi := \varphi_k$, $\varphi_D := (\varphi_k)_D$.  We again
abbreviate $\lambda := \lambda_{\varphi}$ and write $\pi$ and
$\sigma$ for the cuspidal automorphic representations generated
by $\varphi$ and $\Psi$, respectively, and $\pi_D \ni \varphi_D$
for the base change of $\pi$.
We assumed in
\S\ref{sec:notat-prel-reduct} that $\eps(\sigma) = 1$, but
do not impose that assumption here; we had invoked that
assumption above only in \S\ref{sec:choice-quadr-field} and do
not refer here to any results depending upon that section.  We use asymptotic
notation and terminology as in \S\ref{sec:asympt-arch-integr}.
In particular, $\eps > 0$ (resp. $N \in \mathbb{Z}_{\geq 0}$)
are sufficiently small (resp. large) and fixed.
We regard $D$ as fixed,
so that implied constants may freely depend upon it.
The adjectives
``split'' and ``inert'' refer to
$\mathbb{Q}(\sqrt{D})$.
The
discussion of \S\ref{sec:twist-triple-prod} applies to
$\Psi' = \Psi$.

We begin with the twisted analogue of \cite[Lem 2]{MR2680499}.
\begin{lemma}
  We have
  \begin{equation}\label{eqn:bound-for-L-ad-D}
    1/\|\varphi_D\|
    \asymp L(\ad(\pi_D),1)^{-1/2}
    \ll
    (\log \log k)^{\O(1)}
    \exp \sum _{\text{split } p \leq k}
    \frac{ 1 - |\lambda(p)|^2}{p}.
  \end{equation}
\end{lemma}
\begin{proof}
  The first estimate was noted already
  in Lemma \ref{lem:norm-varphi-D},
  so our task is to prove the second estimate.
  
  Recall that $\pi$ is (assumed) non-dihedral.  In particular, the
  adjoint lift $\ad(\pi)$ (and its twist
  $\ad(\pi) \otimes \chi_D$) are cuspidal.  We appeal to the
  factorization\footnote{
    We verify the claimed factorization.
    By the work of Gelbart--Jacquet and Arthur--Clozel,
    each factor in the claimed identity
    is known to be the $L$-function attached
    to an isobaric automorphic representation.
    By strong multiplicity one
    \cite{MR618323},
    it is enough to compare local factors
    at almost all rational primes $p$.
    If $p$ splits $\mathbb{Q}(\sqrt{D})$,
    then both $L_p(\ad(\pi_D),s)$
    and $L_p(\ad(\pi),s) L_p(\ad(\pi) \otimes \chi_D,s)$
    are equal to $L_p(\ad(\pi_p),s)^2$.
    Suppose that $p$ is inert and that $\pi$ and $\chi$ are unramified at $p$.
    Write $E = \mathbb{Q}(\sqrt{D})_p$ for the associated
    quadratic field extension of $\mathbb{Q}_p$.
    Let $\{\alpha, \alpha^{-1} \}$ denote the Satake parameters
    for $\pi$ at $p$.
    Then
    the local component
    at $p$ of 
    $\ad(\pi)$ (resp. $\ad(\pi) \otimes \chi_D$)
    is the unramified representation of $\GL_3(\mathbb{Q}_p)$
    with Satake parameters
    $\{\alpha^2, 1, \alpha^{-2}\}$ (resp. $\{- \alpha^2, -1, -
    \alpha^{-2}\}$).
    The local component of $\pi_D$ (resp. $\ad(\pi_D)$)
    at $p$ 
    is the unramified representation
    of $\GL_2(E)$ (resp. $\GL_3(E)$) with Satake parameters
    $\{\alpha^2, \alpha^{-2}\}$ (resp. $\{\alpha^4, 1,
    \alpha^{-4}\}$).
    The required factorization follows
    from the identity of polynomials
    \[
      (1 - \alpha^4 X) (1 - X) (1 - \alpha^{-4} X)
      =
      (1 - \alpha^2 X) (1 - X) (1 - \alpha^{-2} X)
      \cdot 
      (1 + \alpha^2 X) (1 + X) (1 + \alpha^{-2} X).
    \]
  }
  \begin{equation}\label{eqn:}
    L(\ad(\pi_D),s)
    =
    L(\ad(\pi),s)
    L(\ad(\pi) \otimes \chi_D, s).
  \end{equation}
  In view of the identity $\lambda(p^2) = |\lambda(p)|^2 - 1$,
  the proof of \cite[Lem 2]{MR2680499}
  gives
  \begin{equation}\label{eqn:sound-estimate-L-ad-pi-1}
    L(\ad(\pi),1)
    \gg (\log \log k)^{-3}
    \exp \sum _{p \leq k}
    \frac{|\lambda(p)|^2 - 1}{p}
  \end{equation}
  and
  \begin{equation}\label{eqn:}
    L(\ad(\pi) \otimes \chi_D,s)
    \gg (\log \log k)^{-3}
    \exp \sum _{p \leq k}
    \chi_D(p)
    \frac{|\lambda(p)|^2 - 1}{p},
  \end{equation}
  where by abuse of notation we write $\chi_D(p)$ for the value
  taken at $p$
  by the quadratic Dirichlet
  character corresponding to the idele
  class character $\chi_D$, thus $\chi_D(p) = 0$ unless $p \nmid D$, in which
  case $\chi_D(p) \in \{\pm 1\}$ is the value taken at the
  uniformizer $p \in \mathbb{Q}_p^\times$ of the local component
  $(\chi_D)_p$, which in turn is $+1$ for split $p$ and $-1$ for inert $p$.
  The required conclusion follows from the identity
  \begin{equation}\label{eqn:}
    \frac{1 + \chi_D(p)}{2}
    =
    \begin{cases}
      1 & \text{ if $p$ is split}, \\
      0 & \text{ otherwise.}
    \end{cases}
  \end{equation}
\end{proof}

We turn next
to the twisted analogue
of Soundararajan's estimate \cite[Thm 3]{MR2680499}.
\begin{proposition}\label{prop:twisted-sound}
  We have
  \begin{equation}\label{eqn:}
    \frac{\int  \res(\varphi_D) \Psi}{\|\varphi_D\|}
    \ll
    \exp \sum _{ \text{split } p \leq k}
    \frac{\eps -|\lambda(p)|^2 }{p}.
  \end{equation}
\end{proposition}
\begin{proof}
  By the proof of Proposition \ref{prop:twisted-watson},
  we have
  \begin{equation}\label{eqn:triple-product-upper-bound}
    \frac{|\int  \res(\varphi_D) \Psi|^2}{\|\varphi_D\|^2}
    \ll
    \frac{
      \Lambda(\asai(\pi_D) \times \sigma,1/2)
    }{
      \Lambda(\ad (\pi_D), 1)
    }.
  \end{equation}
  The $\Gamma$-factors appearing on the RHS of
  \eqref{eqn:triple-product-upper-bound} are exactly as in the
  untwisted case considered by Holowinsky--Soundararajan, so by
  an application of Stirling's formula as in
  \cite[p1476]{soundararajan-2008},
  \begin{equation}\label{eqn:}
    \frac{
      \Lambda(\asai(\pi_D) \times \sigma,1/2)
    }{
      \Lambda(\ad (\pi_D), 1)
    }
    \asymp
    \frac{
      k^{-1} L(\asai(\pi_D) \times \sigma,1/2)
    }{
      L(\ad (\pi_D), 1)
    }
  \end{equation}
  We appeal to
  the same consequence of
  Soundararajan's weakly subconvex bounds
  \cite{soundararajan-2008}
  as in the untwisted case:
  \begin{equation}\label{eqn:}
    k^{-1} L(\ad(\pi) \times \sigma,1/2)
    \ll
    (\log k)^{\eps- 1}.
  \end{equation}
  We conclude by \eqref{eqn:bound-for-L-ad-D}
  and the estimate
  $(\log k)^{1/2}
  \asymp
  \exp \sum _{\text{split } p \leq k }
  1/p$.
\end{proof}

We turn to the twisted
analogue of Holowinsky's estimate
 \cite[Thm 2]{MR2680499}.
\begin{proposition}\label{prop:twisted-holowinsky}
  For $k \leq x \leq k^{1+\eps}$ and $|\ell| \leq k^{\eps}$,
  \begin{equation}\label{eqn:twisted-holowinsky}
    x^{-1}
    \sum _{n \leq x}^\sharp
    \left\lvert
      \lambda_D \left(\frac{n + \ell \sqrt{D}}{2} \right)
    \right\rvert
    \ll
    |\ell|^{\O(1)}
    \exp \sum _{\text{split } p \leq x}
    \frac{
      2 (|\lambda(p)| - 1)
      + \eps }{p},
  \end{equation}
  where the notation $\sum^\sharp$ is as in
  the statement of Proposition \ref{prop:asympt-arch-integr}.
\end{proposition}
\begin{proof}
  By summing over $n$ in arithmetic progressions modulo $2 \ell$ and
evaluating $\lambda_D$ as in \S\ref{sec:completion-proof}, we
reduce to verifying that for every irreducible quadratic
polynomial $Q$ with discriminant in the same square-class as
$D$, every nonnegative multiplicative function $f$
bounded by the divisor function,
and all
$x \geq 1$,
we have
\begin{equation}\label{eqn:easy-nair-consequence}
  x^{-1}
  \sum _{n \leq x}
  f(|Q(n)|)
  \ll
  \|Q\|^{\O(1)}
  \exp
  \sum _{\text{split } p \leq x}
  \frac{2( f(p) - 1) + \eps }{p}.
\end{equation}
Such estimates follow readily (in stronger form)
from arguments of Nair
\cite{MR1197420},
but the results stated
by Nair
are not uniform enough to deduce
\eqref{eqn:easy-nair-consequence}.\footnote{We note that we were  likewise unable
to deduce
\eqref{eqn:easy-nair-consequence}
from refinements
of Nair's results
given by Nair--Tenenbaum \cite{MR1618321}
and Henriot \cite{MR2911138,MR3254599}.
The issue in applying the latter work
is that we require $Q$ to be non-primitive,
or equivalently, $f$ to satisfy a weaker condition
than multiplicativity,
which seems to require some uniformity
with respect the parameters ``$A,B$''
in
\cite{MR2911138}.}
For completeness, we record a proof
as in \cite[\S4]{MR2680498}.
We may assume that
$x$ is sufficiently large
and that $\|Q\| \leq x^\eps$,
since otherwise the required estimate is trivial.
We choose $\alpha > 0$ fixed
but sufficiently small in terms of $\eps$,
and set
\begin{equation}\label{eqn:}
  y := x^\alpha,
  \quad
  z := x^{1/\alpha \log \log x}.
\end{equation}
Recall that a natural number
$a$ is called \emph{$z$-smooth}
if every prime divisor $p$ of $a$ satisfies
$p \leq z$.
We say that a natural number $b$ is \emph{$z$-rough}
if every prime divisor $p$ of $b$ satisfies $p > z$.
Given a natural number $n \leq x$, we may factor $|Q(n)|$
uniquely as a product $a b$, where $a$ is $z$-smooth
and $b$ is $z$-rough.  Then
$f(|Q(n)|) = f(a) f(b)$.
Let $\Omega(b)$ denote the number of prime factors of $b$,
counted with multiplicity.
We have $|Q(n)| \leq 3 \|Q\| x^2 \leq 3 x^{2+\eps}$,
so $\Omega(b) \leq \log(|Q(n)|) / \log(z)
\leq \alpha (\log \log x) (\log (3 x^{2+\eps})) / \log x$.
From the assumption that $\alpha$ is small in terms of $\eps$
and the estimate $f(b) \leq 2^{\Omega(b)}$,
it follows that
\[
  f(b) \ll (\log x)^{\eps/2} \asymp \exp \sum _{\text{split }p
    \leq x} \frac{\eps}{p}.
\]
We thereby reduce
to establishing the bound
\begin{equation}\label{eqn:sum-f-of-Q-after-discarding-large-primes}
  x^{-1}
  \sum _{
    \text{$z$-smooth $a \leq x$} 
  }
  f(a) \# N(a)
  \ll
  \|Q\|^{\O(1)}
  \exp
  \sum _{\text{split } p \leq x}
  \frac{2( f(p) - 1) }{p},
\end{equation}
where $N(a)$ denotes the set of all $n \leq x$ admitting a
factorization $|Q(n)| = a b$ as above.
Clearly $\# N(a) \leq x/a + 1 \ll x/a$.
From the estimates
$\sum _{n \leq x}  f(n) \ll x \log(x)$
and
\begin{equation}\label{eqn:}
  \sum _{
    \substack{
      \text{$z$-smooth } a : y < a \leq x \\
    }
  }
  1/a
  \ll
  x/\log (x)^N
\end{equation}
(see, e.g., \cite[(108)]{PDN-HMQUE})
we deduce that the contribution to
\eqref{eqn:sum-f-of-Q-after-discarding-large-primes}
from $a > y$ is negligible for the purposes
of proving \eqref{eqn:sum-f-of-Q-after-discarding-large-primes}.
On the other hand,
for $a \leq y$,
a sieve estimate
as in \cite[p264]{MR1197420}
gives
\begin{equation}\label{eqn:first-bound-for-size-of-N-of-a}
  \# N(a)
  \ll
  x
  \frac{\rho(a)}{a}
  \prod _{p \leq z, p \nmid a}
  \left( 1 - \frac{\rho(p)}{p} \right),
\end{equation}
where $\rho(n)$ denotes the number of roots of
$Q$ in $\mathbb{Z}/n$.
Let $\Xi$ denote the set of primes that either divide the
discriminant
of $Q$
or for which $\rho(p) = p$.
Then for $p \notin \Xi$,
we have
$\rho(p) = 2$ or $0$ according as $p$ is split or inert
in $\mathbb{Q}(\sqrt{D})$.
Discarding wastefully the factors
in \eqref{eqn:first-bound-for-size-of-N-of-a} indexed
by $p \in \Xi$,
we obtain
\begin{equation}\label{eqn:}
  \# N(a)
  \ll
  x
  \prod_{\text{split } p \leq z, p \notin \Xi}
  \left( 1 - \frac{2}{p} \right)
  \frac{\rho(a)}{a}
  \prod _{p | a, p \notin \Xi}
  \left( 1 + \frac{2}{p} \right).
\end{equation}
Thus the LHS of
\eqref{eqn:sum-f-of-Q-after-discarding-large-primes}
is majorized by
\begin{equation}\label{eqn:big-product-p}
  \prod_{\text{split } p \leq z, p \notin \Xi}
  \left( 1 - \frac{2}{p} \right)
  \prod _{p \leq z}
  \left(
    1 +
    \left( 1 + \frac{2}{p} \right)
    \left(
      \frac{\rho(p) f(p)}{p}
      + 
      \frac{\rho(p^2) f(p^2)}{p^2}
      + \dotsb
    \right)
  \right).
\end{equation}
Let $d$ denote the greatest common divisor of the coefficients
of $Q$.
We observe that
each prime divisor of $d$ lies in $\Xi$,
that
$\rho(p^n) \ll \gcd(d,p^n)$
for
$p \in \Xi$ with $p \mid d$,
and that
$\rho(p) \leq 2$ and
$\rho(p^n) \leq 2 \gcd(p^n, D)^{1/2}$
for $p \in \Xi$ with $p \nmid d$;
these last assertions are consequences of Hensel's lemma.
Writing $\tau(d)$ for the number of divisors of $d$,
we deduce that
\eqref{eqn:big-product-p}
is majorized
by
\begin{equation}\label{eqn:}
  \tau(d)^{O(1)}
  \prod _{p | D}
  \left( 1 + \frac{1}{p} \right)^{\O(1)}
  \prod_{\text{split } p \leq z, p \notin \Xi}
  \left( 1 + \frac{2 (f(p) - 1)}{p} \right),
\end{equation}
with an absolute implied constant.
We conclude
by the following crude consequence
of the divisor bound:
\[
  \tau(d)
  \prod _{p | D}
  \left( 1 + \frac{1}{p} \right)
  \ll
  \|Q\|^\eps \ll Q^{\O(1)}. \qedhere
\]
\end{proof}

We pause to record the proof of Theorem \ref{thm:sieve-bound}:
\begin{proof}
  By combining the estimates
  \eqref{eqn:sound-estimate-L-ad-pi-1}
  and
  \eqref{eqn:sum-f-of-Q-after-discarding-large-primes},
  we see that for all $Q = Q_k$ and $f = f_k$ as in the
  hypotheses of Conjecture \ref{conj:non-split-sums},
  \[
    \frac{\sum _{n }
      |\lambda(|Q(n)|)  f(n/k) |
    }{
      k L(\ad \varphi, 1)
    }
    \ll_\eps
    \|Q\|^{\O(1)}
    (\log k)^\eps
    \exp
    \left(
      \sum _{p \leq k}
      \frac{1 - |\lambda(p)|^2}{p}
      + \sum _{\text{split } p \leq k}
      \frac{2 (|\lambda(p)| - 1 )}{p}
    \right).
  \]
  (We refer
  to
  \cite[Proof of Thm 3.1]{PDN-HMQUE} for complete details
  concerning a closely related argument.)
  We have
  \[
    \sum _{p \leq k}
    \frac{1 - |\lambda(p)|^2}{p}
    + \sum _{\text{split } p \leq k}
    \frac{2 (|\lambda(p)| - 1 )}{p}
    =
    \sum _{\text{split } p \leq k}
    \frac{2 |\lambda(p)| - |\lambda(p)|^2}{p}
    -
    \sum _{\text{inert } p \leq k}
    \frac{ |\lambda(p)|^2}{p}
    +
    \O(\log D),
  \]
  with $D$ the discriminant of $Q$ and with ``split'' and ``inert''
  referring to $\mathbb{Q}(\sqrt{D})$.  Using the trivial lower
  bound $|\lambda(p)|^2 \geq 0$ for inert $p$ and the upper bound
  $2 |\lambda(p)| - |\lambda(p)|^2 \leq 1$ for split $p$, we
  deduce that the above is at most
  $(1/2) \log \log k + \O(\log D)$.  Since
  $D \ll \|Q\|^{\O(1)}$, the required bound
  \eqref{eqn:sieve-bound-second-estimate}
  follows.
\end{proof}

Returning
to the proof of
Theorem
\protect\hyperlink{thm:twisted-holowinsky-sound}{B},
we note the following consequence of the preceding estimates:
\begin{corollary}
  \begin{equation}\label{eqn:triple-product-upper-bound-2}
      \frac{\int  \res(\varphi_D) \Psi}{\|\varphi_D\|}
      \ll
      \exp
      \left(
        \sum _{\text{split } p \leq x}
        \frac{\eps - (1 - |\lambda(p)|)^2}{p}
      \right).
  \end{equation}
\end{corollary}
\begin{proof}
  We insert the estimate \eqref{eqn:twisted-holowinsky} into the
  asymptotic formula \eqref{eqn:sum-just-after-truncation} for
  $\int \res(\varphi_D) \Psi$.
  In view of the decay properties of the test functions $f_\ell$
  occurring in that formula,
  we obtain
  \begin{equation}\label{eqn:}
    \int \res(\varphi_D) \Psi
    \ll
    \exp \sum _{\text{split } p \leq k}
    \frac{
      2 (|\lambda(p)| - 1)
      + \eps }{p}
  \end{equation}
  (see \cite[Proof of Thm 3.1]{PDN-HMQUE} for complete details
  concerning a closely related argument).  We then appeal to
  \eqref{eqn:bound-for-L-ad-D} and the identity
  \begin{equation}\label{eqn:}
    2 (|\lambda(p)| - 1) + (1 - |\lambda(p)|^2)
    = - (1 - |\lambda(p)|)^2
  \end{equation}
  (compare with \cite[(1.2)]{MR2680499}).
\end{proof}

We now complete the proof of Theorem
\hyperlink{thm:twisted-holowinsky-sound}{B}.
By the Deligne bound $|\lambda(p)| \leq 2$,
we may write
\begin{equation}\label{eqn:}
  \frac{
    \sum _{\text{split p} \leq k}
    |\lambda(p)|^2/p
  }{
    \sum _{\text{split p} \leq k}
    1/p
  }
  = c^2
\end{equation}
for some $c \in [0,2]$.
It follows then from the
Minkowski inequality
(i.e., the triangle inequality for $\ell^2(\mathbb{Z})$)
that
\begin{equation}\label{eqn:}
  \frac{
    \sum _{\text{split p} \leq k}
    (1 -  |\lambda(p)|)^2/p
  }{
    \sum _{\text{split p} \leq k}
    1/p
  }
  \geq
  (1 -c)^2,
\end{equation}
hence
by
Proposition \ref{prop:twisted-sound},
Proposition \ref{prop:twisted-holowinsky}
and the estimate
$\exp \sum_{\text{split }p \leq k} 1/p \asymp (\log
k)^{1/2}$
that
\begin{equation}\label{eqn:}
  \frac{\int  \res(\varphi_D) \Psi}{\|\varphi_D\|}
  \ll
  (\log k)^{-\max(c^2, (1-c)^2)/2 + \eps}.
\end{equation}
The quantity $\max(c^2, (1 - c)^2)$ is
minimized when $c^2 = (1 - c)^2$, i.e., for $c=1/2$, in which
case $\max(c^2,(1-c)^2)/2 = 1/8$.  The proof of Theorem
\hyperlink{thm:twisted-holowinsky-sound}{B} is thus complete.

\begin{remark}
  In the split case, it seems likely that similar arguments
  yield for the LHS of \eqref{eqn:hque-conj} the estimate
  $\ll (\log k)^{-\delta + \eps}$ with
  $\delta = \min_{c \in [0,2]} \max(c^2-1/2,(1-c)^2) = 1/16$
  (for $\Psi$ cuspidal), improving upon the exponent
  $\delta = 1/30$ obtained in \cite[Thm 1 (i)]{MR2680499}.
\end{remark}

\subsection*{Acknowledgements}
We gratefully acknowledge the support of NSF grant OISE-1064866
and SNF grant SNF-137488 during the work leading to this paper.

\def\cprime{$'$} \def\cprime{$'$} \def\cprime{$'$} \def\cprime{$'$}

% -----------------

\begin{thebibliography}{10}

\bibitem[Bl]{MR2435749}
Valentin Blomer.
\newblock Sums of {H}ecke eigenvalues over values of quadratic polynomials.
\newblock {\em Int. Math. Res. Not. IMRN}, (16):Art. ID rnn059. 29, 2008.

\bibitem[BB]{MR2811610}
Valentin Blomer and Farrell Brumley.
\newblock On the {R}amanujan conjecture over number fields.
\newblock {\em Ann. of Math. (2)}, 174(1):581--605, 2011.

\bibitem[BH]{MR2437682}
Valentin Blomer and Gergely Harcos.
\newblock The spectral decomposition of shifted convolution sums.
\newblock {\em Duke Math. J.}, 144(2):321--339, 2008.

\bibitem[BSS]{MR2018269}
Siegfried B{\"o}cherer, Peter Sarnak, and Rainer Schulze-Pillot.
\newblock Arithmetic and equidistribution of measures on the sphere.
\newblock {\em Comm. Math. Phys.}, 242(1-2):67--80, 2003.

\bibitem[BH]{MR2234120}
Colin~J. Bushnell and Guy Henniart.
\newblock {\em The local {L}anglands conjecture for {$\rm GL(2)$}}, volume 335
  of {\em Grundlehren der Mathematischen Wissenschaften [Fundamental Principles
  of Mathematical Sciences]}.
\newblock Springer-Verlag, Berlin, 2006.

\bibitem[Cas]{MR0337789}
William Casselman.
\newblock On some results of {A}tkin and {L}ehner.
\newblock {\em Math. Ann.}, 201:301--314, 1973.

\bibitem[CC]{MR4033818}
Shih-Yu {Chen} and Yao {Cheng}.
\newblock {On Deligne's conjecture for certain automorphic $L$-functions for
  ${\rm GL}(3)\times {\rm GL}(2)$ and ${\rm GL}(4)$}.
\newblock {\em Documenta Mathematica}, 24:2241--2297, 2019.

\bibitem[Ch]{MR4193477}
Yao Cheng.
\newblock Special value formula for the twisted triple product {$L$}-function
  and an application to the restricted {$L^2$}-norm problem.
\newblock {\em Forum Math.}, 33(1):59--108, 2021.

\bibitem[Coh]{MR3889557}
Henri Cohen.
\newblock Expansions at cusps and {P}etersson products in {P}ari/{GP}.
\newblock In {\em Elliptic integrals, elliptic functions and modular forms in
  quantum field theory}, Texts Monogr. Symbol. Comput., pages 161--181.
  Springer, Cham, 2019.

\bibitem[Col]{2018arXiv180209740C}
Dan~J. {Collins}.
\newblock {Numerical computation of Petersson inner products and
  $q$-expansions}.
\newblock {\em arXiv e-prints}, page arXiv:1802.09740, Feb 2018.

\bibitem[CI]{MR1779567}
J.~B. Conrey and H.~Iwaniec.
\newblock The cubic moment of central values of automorphic {$L$}-functions.
\newblock {\em Ann. of Math. (2)}, 151(3):1175--1216, 2000.

\bibitem[FH]{MR1343325}
Solomon Friedberg and Jeffrey Hoffstein.
\newblock Nonvanishing theorems for automorphic {$L$}-functions on {${\rm
  GL}(2)$}.
\newblock {\em Ann. of Math. (2)}, 142(2):385--423, 1995.

\bibitem[Gan]{MR2438071}
Wee~Teck Gan.
\newblock Trilinear forms and triple product epsilon factors.
\newblock {\em Int. Math. Res. Not. IMRN}, (15):Art. ID rnn058, 15, 2008.

\bibitem[GL]{MR546613}
P.~G\'erardin and J.-P. Labesse.
\newblock The solution of a base change problem for {${\rm GL}(2)$}\ (following
  {L}anglands, {S}aito, {S}hintani).
\newblock In {\em Automorphic forms, representations and {$L$}-functions
  ({P}roc. {S}ympos. {P}ure {M}ath., {O}regon {S}tate {U}niv., {C}orvallis,
  {O}re., 1977), {P}art 2}, Proc. Sympos. Pure Math., XXXIII, pages 115--133.
  Amer. Math. Soc., Providence, R.I., 1979.

\bibitem[Go]{MR3889963}
Roger Godement.
\newblock {\em Notes on {J}acquet-{L}anglands' theory}, volume~8 of {\em CTM.
  Classical Topics in Mathematics}.
\newblock Higher Education Press, Beijing, 2018.
\newblock With commentaries by Robert Langlands and Herve Jacquet.

\bibitem[He1]{MR2911138}
Kevin Henriot.
\newblock Nair-{T}enenbaum bounds uniform with respect to the discriminant.
\newblock {\em Math. Proc. Cambridge Philos. Soc.}, 152(3):405--424, 2012.

\bibitem[He2]{MR3254599}
Kevin Henriot.
\newblock Nair-{T}enenbaum uniform with respect to the discriminant---{E}rratum
  [mr2911138].
\newblock {\em Math. Proc. Cambridge Philos. Soc.}, 157(2):375--377, 2014.

\bibitem[HL]{HL94}
Jeffrey Hoffstein and Paul Lockhart.
\newblock Coefficients of {M}aass forms and the {S}iegel zero.
\newblock {\em Ann. of Math. (2)}, 140(1):161--181, 1994.
\newblock With an appendix by Dorian Goldfeld, Hoffstein and Daniel Lieman.

\bibitem[Ho1]{MR2484279}
Roman Holowinsky.
\newblock A sieve method for shifted convolution sums.
\newblock {\em Duke Math. J.}, 146(3):401--448, 2009.

\bibitem[Ho2]{MR2680498}
Roman Holowinsky.
\newblock Sieving for mass equidistribution.
\newblock {\em Ann. of Math. (2)}, 172(2):1499--1516, 2010.

\bibitem[HS]{MR2680499}
Roman Holowinsky and Kannan Soundararajan.
\newblock Mass equidistribution for {H}ecke eigenforms.
\newblock {\em Ann. of Math. (2)}, 172(2):1517--1528, 2010.

\bibitem[Hu]{MR3801500}
Yueke Hu.
\newblock Triple product formula and mass equidistribution on modular curves of
  level {$N$}.
\newblock {\em Int. Math. Res. Not. IMRN}, (9):2899--2943, 2018.

\bibitem[Ic1]{MR2198222}
Atsushi Ichino.
\newblock Pullbacks of {S}aito-{K}urokawa lifts.
\newblock {\em Invent. Math.}, 162(3):551--647, 2005.

\bibitem[Ic2]{MR2449948}
Atsushi Ichino.
\newblock Trilinear forms and the central values of triple product
  {$L$}-functions.
\newblock {\em Duke Math. J.}, 145(2):281--307, 2008.

\bibitem[Iw]{IwaniecQUENotes}
Henryk Iwaniec.
\newblock Notes on the quantum unique ergodicity for holomorphic cusp forms,
  2010.

\bibitem[JPSS]{MR620708}
H.~Jacquet, I.~I. Piatetski-Shapiro, and J.~Shalika.
\newblock Conducteur des repr\'esentations du groupe lin\'eaire.
\newblock {\em Math. Ann.}, 256(2):199--214, 1981.

\bibitem[JS]{MR618323}
H.~Jacquet and J.~A. Shalika.
\newblock On {E}uler products and the classification of automorphic
  representations. {I} and {II}.
\newblock {\em Amer. J. Math.}, 103(3):499--558 and 777--815, 1981.

\bibitem[Ki]{MR1937203}
Henry~H. Kim.
\newblock Functoriality for the exterior square of {${\rm GL}_4$} and the
  symmetric fourth of {${\rm GL}_2$}.
\newblock {\em J. Amer. Math. Soc.}, 16(1):139--183 (electronic), 2003.
\newblock With appendix 1 by Dinakar Ramakrishnan and appendix 2 by Kim and
  Peter Sarnak.

\bibitem[Li]{MR2195133}
Elon Lindenstrauss.
\newblock Invariant measures and arithmetic quantum unique ergodicity.
\newblock {\em Ann. of Math. (2)}, 163(1):165--219, 2006.

\bibitem[LS]{luo-sarnak-mass}
Wenzhi Luo and Peter Sarnak.
\newblock Mass equidistribution for {H}ecke eigenforms.
\newblock {\em Comm. Pure Appl. Math.}, 56(7):874--891, 2003.
\newblock Dedicated to the memory of J{\"u}rgen K. Moser.

\bibitem[Mi]{MR2331346}
Philippe Michel.
\newblock Analytic number theory and families of automorphic {$L$}-functions.
\newblock In {\em Automorphic forms and applications}, volume~12 of {\em
  IAS/Park City Math. Ser.}, pages 181--295. Amer. Math. Soc., Providence, RI,
  2007.

\bibitem[MV]{michel-2009}
Philippe Michel and Akshay Venkatesh.
\newblock The subconvexity problem for {${\rm GL}_2$}.
\newblock {\em Publ. Math. Inst. Hautes \'Etudes Sci.}, (111):171--271, 2010.

\bibitem[Na]{MR1197420}
Mohan Nair.
\newblock Multiplicative functions of polynomial values in short intervals.
\newblock {\em Acta Arith.}, 62(3):257--269, 1992.

\bibitem[NT]{MR1618321}
Mohan Nair and G{\'e}rald Tenenbaum.
\newblock Short sums of certain arithmetic functions.
\newblock {\em Acta Math.}, 180(1):119--144, 1998.

\bibitem[Ne1]{PDN-HQUE-LEVEL}
Paul~D. Nelson.
\newblock Equidistribution of cusp forms in the level aspect.
\newblock {\em Duke Math. J.}, 160(3):467--501, 2011.

\bibitem[Ne2]{PDN-HMQUE}
Paul~D. Nelson.
\newblock Mass equidistribution of {H}ilbert modular eigenforms.
\newblock {\em The Ramanujan Journal}, 27:235--284, 2012.

\bibitem[Ne3]{MR3356036}
Paul~D. Nelson.
\newblock Evaluating modular forms on {S}himura curves.
\newblock {\em Math. Comp.}, 84(295):2471--2503, 2015.

\bibitem[Ne4]{nelson-variance-73-2}
Paul~{D}. Nelson.
\newblock Quantum variance on quaternion algebras, {I}.
\newblock preprint, 2016.

\bibitem[Ne5]{nelson-variance-II}
Paul~{D}. Nelson.
\newblock Quantum variance on quaternion algebras, {II}.
\newblock preprint, 2017.

\bibitem[Ne6]{Nelson-EisCubic}
Paul~D. Nelson.
\newblock Eisenstein series and the cubic moment for $\text{PGL}_2$.
\newblock arxiv preprint, 2019.

\bibitem[Ne7]{nelson-variance-3}
Paul~{D}. Nelson.
\newblock Quantum variance on quaternion algebras, {III}.
\newblock preprint, 2019.

\bibitem[Ne8]{nelson-subconvex-reduction-eisenstein}
Paul~D. Nelson.
\newblock Subconvex equidistribution of cusp forms: {R}eduction to {E}isenstein
  observables.
\newblock {\em Duke Math. J.}, 168(9):1665--1722, 2019.

\bibitem[NPS]{MR3110797}
Paul~D. Nelson, Ameya Pitale, and Abhishek Saha.
\newblock Bounds for {R}ankin--{S}elberg integrals and quantum unique
  ergodicity for powerful levels.
\newblock {\em J. Amer. Math. Soc.}, 27(1):147--191, 2014.

\bibitem[Pe]{MR3394377}
Ian~N. Petrow.
\newblock A twisted {M}otohashi formula and {W}eyl-subconvexity for
  {$L$}-functions of weight two cusp forms.
  \newblock {\em Math. Ann.}, 363(1-2):175--216, 2015.








\bibitem[Pr1]{MR1059954}
Dipendra Prasad.
\newblock Trilinear forms for representations of {${\rm GL}(2)$} and local
  {$\epsilon$}-factors.
\newblock {\em Compositio Math.}, 75(1):1--46, 1990.

\bibitem[Pr2]{MR1198305}
Dipendra Prasad.
\newblock Invariant forms for representations of {${\rm GL}_2$} over a local
  field.
\newblock {\em Amer. J. Math.}, 114(6):1317--1363, 1992.

\bibitem[PS]{MR2369492}
Dipendra Prasad and Rainer Schulze-Pillot.
\newblock Generalised form of a conjecture of {J}acquet and a local
  consequence.
\newblock {\em J. Reine Angew. Math.}, 616:219--236, 2008.


\bibitem[Q]{MR3291638}
Yannan Qiu.
\newblock Periods of {S}aito-{K}urokawa representations.
\newblock {\em Int. Math. Res. Not. IMRN}, (24):6698--6755, 2014.

\bibitem[RS]{MR1266075}
Ze{\'e}v Rudnick and Peter Sarnak.
\newblock The behaviour of eigenstates of arithmetic hyperbolic manifolds.
\newblock {\em Comm. Math. Phys.}, 161(1):195--213, 1994.

\bibitem[Sa]{Sar01}
Peter Sarnak.
\newblock Estimates for {R}ankin-{S}elberg {$L$}-functions and quantum unique
  ergodicity.
\newblock {\em J. Funct. Anal.}, 184(2):419--453, 2001.

\bibitem[Sch]{Sch02}
Ralf Schmidt.
\newblock Some remarks on local newforms for {$\rm GL(2)$}.
\newblock {\em J. Ramanujan Math. Soc.}, 17(2):115--147, 2002.

\bibitem[So1]{2009arXiv0901.4060S}
Kannan Soundararajan.
\newblock Quantum unique ergodicity for {${\rm SL}_2(\mathbb{Z}) \backslash
  \mathbb{H} $}.
\newblock {\em Ann. of Math. (2)}, 172(2):1529--1538, 2010.

\bibitem[So2]{soundararajan-2008}
Kannan Soundararajan.
\newblock Weak subconvexity for central values of {$L$}-functions.
\newblock {\em Ann. of Math. (2)}, 172(2):1469--1498, 2010.

\bibitem[Te]{MR2473297}
Nicolas Templier.
\newblock Minoration de rangs de courbes elliptiques.
\newblock {\em C. R. Math. Acad. Sci. Paris}, 346(23-24):1225--1230, 2008.

\bibitem[TT]{MR3096570}
Nicolas Templier and Jacob Tsimerman.
\newblock Non-split sums of coefficients of {$GL(2)$}-automorphic forms.
\newblock {\em Israel J. Math.}, 195(2):677--723, 2013.

\bibitem[Wa]{watson-2008}
Thomas~Crawford Watson.
\newblock {\em Rankin triple products and quantum chaos}.
\newblock ProQuest LLC, Ann Arbor, MI, 2002.
\newblock Thesis (Ph.D.)--Princeton University, \url{arXiv.org:0810.0425}.


\end{thebibliography}
\end{document}